\documentclass[12pt]{article}

\usepackage{amsmath,amssymb,amsthm,color,CJK,mathrsfs}

\usepackage{cases}

\newcommand{\lc}{\mathrel{\raise2pt\hbox{${\mathop<\limits_{\raise1pt\hbox{\mbox{$\sim$}}}}$}}}
\newcommand{\gc}{\mathrel{\raise2pt\hbox{${\mathop>\limits_{\raise1pt\hbox{\mbox{$\sim$}}}}$}}}

\usepackage{anysize}
\usepackage{setspace}
\usepackage{caption}
\usepackage{enumerate}
\usepackage{indentfirst}
\usepackage{diagbox}
\usepackage{float}
\usepackage{graphicx}
\usepackage{graphics}
\usepackage{abstract}
\usepackage{algorithmicx}
\usepackage[noend]{algpseudocode}
\usepackage{subfigure}
\usepackage{epsfig}
\usepackage{multirow}
\usepackage{threeparttable}

\usepackage{tikz}
\usetikzlibrary{arrows,shapes,chains,matrix}

\usepackage{mathtools}

\newtheorem{thm}{Theorem}[section]
\newtheorem{lem}{Lemma}[section]
\newtheorem{remark}{Remark}[section]

\newtheorem{algorithm}{Algorithm}[section]
\numberwithin{equation}{section}

\begin{document}
\title{Symmetrized two-scale finite element discretizations for partial differential equations with symmetric solutions}

\author{Pengyu Hou\thanks{School of Mathematical Sciences, Beijing Normal University, Beijing 100875, China ({hou.pengyu@mail.bnu.edu.cn}).}
\and Fang Liu\thanks{School of Statistics and Mathematics, Central University of Finance and Economics, Beijing 102206, China ({fliu@cufe.edu.cn}).}
\and Aihui Zhou\thanks{LSEC, Institute of Computational Mathematics and Scientific/Engineering Computing, Academy of Mathematics and Systems Science, Chinese Academy of Sciences,
Beijing 100190, China ({azhou@lsec.cc.ac.cn})}.}

\maketitle

\begin{abstract}
In this paper, a symmetrized two-scale finite element method is proposed for a class of partial differential equations with symmetric solutions. With this method, the finite element approximation on a fine tensor product grid is reduced to the finite element approximations on a much coarse grid and a univariant fine grid. It is shown by both theory and numerics including electronic structure calculations that the resulting approximation still maintains an asymptotically optimal accuracy. Consequently the symmetrized two-scale finite element method reduces computational cost significantly.
\end{abstract}



\section{Introduction}
To improve the efficiency of solving multi-dimensional elliptic source and eigenvalue problems on tensor-product domains, the two-scale finite element method was proposed in \cite{GLZ08,LZ06,LZ07}. With this method, low frequency parts of the finite element solution to an elliptic problem are captured on a coarse grid and high frequency components are handled by some univariant fine and coarse grids. 

We understand that many quantities in science and engineering have symmetries such as wavefunctions, Hamiltonians, and interatomic potentials in quantum physics. It is shown in literature that efficient numerical methods can be designed to reduce the computational cost of solving problems with symmetric properties (see, e.g., \cite{Bachmayr2021,Fang2013,Han2022} and references cited therein). For instance, a symmetry-based decomposition approach was proposed to solve eigenvalue problems with spatial symmetries in \cite{Fang2013}. By this approach, the original differential eigenvalue problem is decomposed into some eigenvalue subproblems that can be carried out in parallel. In addition, each subproblem requires only a smaller number of eigenpairs compared with the original problem. As a result, the computational cost is reduced.

To reduce the computational cost of approximations of the elliptic source and eigenvalue problems with symmetric solutions, in this paper, we study some symmetrized two-scale finite element discretizations.

Let us give an informal description of the main ideas and results in this paper. Set $k$ be a positive integer. The symmetric group $\text{Sym}(k)$ is the set of bijections of $\{1,2,\cdots,k\}$ to itself. Specially, a permutation $\sigma \in \text{Sym}(k)$ which interchanges two letters $i$ and $j$ and leaves all the other letters unchanged is called a transposition. Denote the
transposition $\sigma$ which interchanges $i$ and $j$ by $(i,j)$ \cite{Joyner2008}.

Set $\Omega=(0,1)^3$ and $\mathbf{x}=(x_1,x_2,x_3)$. We say that $u$ is a symmetric function on $\bar\Omega$  \cite{Bachmayr2021} if
\begin{equation*}
u(x_{\sigma(1)},x_{\sigma(2)},x_{\sigma(3)})=u(x_1,x_2,x_3), \quad \quad \forall \mathbf{x}\in \bar\Omega,~\sigma\in \text{Sym}(3).
\end{equation*}
Let ${T}^{h_1,h_2,h_3}(\Omega)$ be a uniform tensor product grid on $\bar\Omega$. $S^{h_1,h_2,h_3}(\Omega)$ is the associated tensor-product space of piecewise trilinear functions on $\bar{\Omega}$. Let $u_{h_1,h_2,h_3}\in S^{h_1,h_2,h_3}(\Omega)$ be a finite element approximation to $u$.

In computation, the grid points of $T^{h_1,h_2,h_3}(\Omega)$ are usually numbered along the $x_1$-direction, the $x_2$-direction, and the $x_3$-direction, consecutively. Let $N_i=1/h_i (i=1,2,3)$ be a positive integer. Then the values of $u_{h_1,h_2,h_3}$ on these grid points are stored in a vector denoted by $U^{h_1,h_2,h_3}$ with the size $N_{dof}:=\prod_{i=1}^3 (N_i+1)$. Let $G: \mathbb{C}^{N_{dof}}\rightarrow S^{h_1,h_2,h_3}(\Omega)$ be a bijection satisfying
\begin{equation*}
G(U^{h_1,h_2,h_3})=\sum_{i=1}^{N_{dof}} U^{h_1,h_2,h_3}_i \phi_i, \quad \quad \forall  U^{h_1,h_2,h_3}\in \mathbb{C}^{N_{dof}},
\end{equation*}
where $U^{h_1,h_2,h_3}_i$ is the $i$-th component of $U^{h_1,h_2,h_3}$ and $\phi_i$ is the Lagrangian basis function of $S^{h_1,h_2,h_3}(\Omega)$ with $i=1,2,\cdots,N_{dof}$. Obviously, we have $u_{h_1,h_2,h_3}=G(U^{h_1,h_2,h_3})$.

For any $\sigma \in \text{Sym}(3)$, there holds $u_{h_{\sigma(1)},h_{\sigma(2)},h_{\sigma(3)}}=G(U^{h_{\sigma(1)},h_{\sigma(2)},h_{\sigma(3)}})$. We define a vector transformation $\mathcal{T}_{\sigma}: \mathbb{C}^{N_{dof}}\rightarrow \mathbb{C}^{N_{dof}}$ satisfying
\begin{equation*}
\mathcal{T}_{\sigma}(U^{h_1,h_2,h_3})=U^{h_{\sigma(1)},h_{\sigma(2)},h_{\sigma(3)}},\quad \forall U^{h_1,h_2,h_3}\in  \mathbb{C}^{N_{dof}}.
\end{equation*}
Then $U^{h_{\sigma(1)},h_{\sigma(2)},h_{\sigma(3)}}$ can be obtained from $U^{h_1,h_2,h_3}$ by a symmetrization operation (Algorithm \ref{alg:symmetric-dd}). Let $R_{\sigma}: S^{h_1,h_2,h_3}(\Omega)\rightarrow S^{h_{\sigma(1)},h_{\sigma(2)},h_{\sigma(3)}}(\Omega)$ be a bijection defined by
\begin{equation}\label{eq:R-def}
R_{\sigma}=G\circ \mathcal{T}_{\sigma} \circ G^{-1}.
\end{equation}
We get
\begin{equation}\label{eq:trans-3d}
u_{h_{\sigma(1)},h_{\sigma(2)},h_{\sigma(3)}} =R_{\sigma}(u_{h_1,h_2,h_3}),\quad \forall u_{h_1,h_2,h_3}\in S^{h_1,h_2,h_3}(\Omega).
\end{equation}

Note that the two-scale finite element solution \cite{GLZ08,LZ06} of a linear elliptic source problem with the Dirichlet boundary condition is a linear combination of the standard finite element solutions on four different scale grids. That is,
$$
B^h_{H,H,H}u=u_{h,H,H}+u_{H,h,H}+u_{H,H,h}-2u_{H,H,H}.
$$
Hence, for a linear source problem with symmetric solution, we obtain from \eqref{eq:trans-3d} that $u_{H,h,H}=R_{(1,2)}(u_{h,H,H})$ and $u_{H,H,h}=R_{(1,3)}(u_{h,H,H})$. Consequently,
$$
B^h_{H,H,H}u=u_{h,H,H}+R_{(1,2)}(u_{h,H,H})+R_{(1,3)}(u_{h,H,H})-2u_{H,H,H}.
$$
It is seen from \eqref{eq:R-def} that the amount of operations for getting $R_{(1,2)}(u_{h,H,H})$ or $R_{(1,3)}(u_{h,H,H})$ from $u_{h,H,H}$ is essentially that for calculating $\mathcal{T}_{(1,2)}(U^{h,H,H})$ or $\mathcal{T}_{(1,3)}(U^{h,H,H})$ from $U^{h,H,H}$ by the symmetrization algorithm (Algorithm \ref{alg:symmetric-dd}), respectively. It is only $18 N_{dof}$ with $N_{dof}:=(n-1)(N-1)(N-1)$ if $n=1/h$ and $N=1/H$. While the amount of operations for computing $u_{H,h,H}$ or $u_{H,H,h}$ by the standard finite element method is far from $18 N_{dof}$, due to creating stiffness matrix and solving linear algebraic systems. Hence the symmetrization algorithm is quite efficient. Similarly, the symmetrization algorithm can be also combined with the postprocessed two-scale finite element method \cite{LSZ11} and the two-scale postprocessed finite element method \cite{LZ16} respectively to construct new and more efficient algorithms.

This paper is structured as follows. In section 2, some preliminaries are presented and the existing related results that will be used in this paper are called. In section 3, a vector transformation operator associated with symmetric functions is constructed. A symmetrization algorithm is then proposed to perform this vector transformation operator. In section 4 and section 5, combining the symmetrization algorithm with the two-scale finite element method, some symmetrized two-scale finite element discretizations are developed for the linear source and eigenvalue problems with symmetric solutions, respectively. In section 6, several numerical results, including the applications to electronic structure calculations, are presented to illustrate the efficiency of our approaches. Finally, some concluding remarks are given in section 7.

\section{Preliminaries}

Let $\Omega=(0,1)^d$ with $d\ge 2$. The standard notation for Sobolev spaces $W^{s,p}(\Omega)$ and their associated norms and seminorms will be used \cite{Cia78}. Set $H^s(\Omega)=W^{s,2}(\Omega)$ and $\lVert\ \cdot \ \rVert_{s,\Omega}=\lVert \ \cdot \ \rVert_{s,2,\Omega}$. Let $H_0^1(\Omega) = \{ v\in H^1(\Omega):  v\mid_{\partial \Omega} = 0 \}$, $H^{-1}(\Omega)$ be the dual of $H_0^1(\Omega)$, and $(\cdot,\cdot)$ be the standard $L^2(\Omega)$ inner product. We shall use $A\lesssim B$ to mean that $A\le CB$ for some constant $C$ which does not depend on any grid parameters.

Let $\mathbb{N}_0$ be the set of all nonnegative integers. For $w\in W^{s,p}(\Omega)$, $\mathbf{x}=(x_1,x_2,\cdot \cdot \cdot,x_d)\in \Omega$, and $\mathbf{\alpha} = (\alpha_1,\alpha_2,\cdot \cdot \cdot,\alpha_d)\in \mathbb{N}_0^d$, set
    \begin{equation*}
        (D^\mathbf{\alpha} w)(\mathbf{x})=(\frac{\partial^{\alpha_1}}{\partial x_1^{\alpha_1}}\cdot\cdot\cdot\frac{\partial^{\alpha_d}w}{\partial x_d^{\alpha_d}})(\mathbf{x}) \quad\quad \mbox{and}\quad\quad   \lvert \mathbf{\alpha} \rvert=\alpha_1+\alpha_2+\cdot\cdot\cdot+\alpha_d.
    \end{equation*}

Furthermore, set $\textbf{0}=(0,\cdot\cdot\cdot,0)\in\mathbb{R}^d$, $\textbf{e}=(1,\cdot\cdot\cdot,1)\in\mathbb{R}^d$, and $\mathbb{Z}_d={1,2,\cdot\cdot\cdot,d}$. For each $i\in\mathbb{Z}_d$, let $\textbf{e}_i=(0,\cdot\cdot\cdot,0,1,0,\cdot\cdot\cdot,0)\in\mathbb{R}^d$, which is the vector whose $i$-th component is one and other components are zero. Write $\hat{\textbf{e}}_i=\textbf{e}-\textbf{e}_i$.

We shall use the following Sobolev space
    \begin{equation*}
        W^{G,4}(\Omega)=\{w\in H^{3}(\Omega):D^\mathbf{\alpha} w\in L^2(\Omega), \mathbf{0}\le \mathbf{\alpha} \le 3\mathbf{e},\lvert\mathbf{\alpha}\rvert=4\}
    \end{equation*}
with its natural norm $\lVert\ \cdot\ \rVert_{W^{G,4}(\Omega)}$ \cite{PZ99}.

\subsection{A source problem}

Consider the following elliptic source problem:
\begin{eqnarray}\label{equ:boundary_value_pro}
    \left\{\begin{aligned}
     L u &= f \ \ \mbox{in}\ \ \Omega,\\
     u &= 0\ \ \mbox{on}\ \  \partial\Omega.
    \end{aligned}\right.
\end{eqnarray}
Here $f\in H^{-1}(\Omega)$ and $L$ is defined by
\begin{equation*}
    Lu = -\sum_{i,j=1}^{d}\frac{\partial}{\partial x_i}\big(a_{ij}\frac{\partial u}{\partial x_j}\big)+\sum_{i=1}^{d}b_i\frac{\partial u}{\partial x_i}+Vu,
\end{equation*}
where $a_{ij}\in W^{1,\infty}(\Omega),\ b_i,\ V\in L^{\infty}(\Omega)$, and $\sum_{i,j=1}^d a_{ij}x_ix_j\ge c\sum_{k=1}^d x_k^2 $ on $\bar{\Omega}$ for some constant $c > 0$.

A weak form of (\ref{equ:boundary_value_pro}) is: find $u\in H_0^1(\Omega)$ such that
\begin{equation}\label{equ:weak_form_boundary_value_pro}
    a(u,v) = (f,v)\ \ \ \ \forall v\in H_0^1(\Omega),
\end{equation}
where
\begin{equation}\label{equ:a_uv}
    a(u,v) = \int_\Omega\Big(\sum_{i,j=1}^{d}a_{ij}\frac{\partial u}{\partial x_i}\frac{\partial v}{\partial x_j}+\sum_{i=1}^{d}b_i\frac{\partial u}{\partial x_i}v+Vuv\Big) \quad \mbox{and} \quad  (f,v)=\int_\Omega fv.
\end{equation}
For each $f\in H^{-1}(\Omega)$, assume that there is a unique solution $u\in
H_0^1(\Omega)$ of \eqref{equ:weak_form_boundary_value_pro}.

Set $N\in\mathbb{N}_0 /\{0\}$ and $h = 1/N\in [0,1]$. Let $T^{h}[0,1]$ be the uniform grid with grid size $h$. Let $h_j=1/N_j$ for $N_j\in\mathbb{N}_0/\{0\},\ j\in\mathbb{Z}_d$. Set ${\bf h}=(h_1,\ldots,h_{d})$ and $\displaystyle h =\max_{j\in \mathbb{Z}_d} h_j$. Define a tensor-product grid on $\bar\Omega=[0,1]^d$ by
$$
T^{\bf h}(\Omega):= T^{h_1}[0,1]\times\cdots\times T^{h_{d}}[0,1].
$$
Let $S^{h}[0,1] \subset C[0,1]$ be the piecewise linear finite element
space on $T^h[0,1]$.
The tensor product spaces of piecewise $d$-linear
functions on $\bar\Omega$ are
$$
S^{\bf h}(\Omega) :=S^{h_1}[0,1]\otimes\cdots\otimes S^{h_{d}}[0,1]
\qquad\text{and}\qquad S_0^{\bf h}(\Omega) := S^{\bf h}(\Omega)\cap
H_0^1(\Omega).
$$

The standard Galerkin finite element method for solving (\ref{equ:weak_form_boundary_value_pro}) is: find $u_{\textbf{h}}\in S_0^{\textbf{h}}(\Omega)$ such that
\begin{equation}\label{equ:boundary_FEM}
    a(u_{\textbf{h}},v) = (f,v)\ \ \ \forall \ v\in S_0^{\textbf{h}}(\Omega).
\end{equation}
Define the Galerkin projector $P_{\bf h}:
H^1_0(\Omega)\mapsto S^{\bf h}_0(\Omega)$ by
\begin{equation}\label{Gprojection}
  {a}(w-P_{\bf h} w, v) =0 ~\quad \forall v\in S^{\bf h}_0(\Omega),\quad \forall w\in H^1_0(\Omega).
\end{equation}
Obviously $P_{\bf h}w=w_{\bf h},~\forall w\in H^1_0(\Omega)$.

\subsection{An eigenvalue problem}

Assume that $a(\cdot,\cdot)$ is the same as \eqref{equ:a_uv}
with $b_i=0, ~i=1,2,\cdots,d$ and $(a_{ij})$ is symmetric. Consider the following eigenvalue problem
\begin{eqnarray}\label{equ:weak_eigenvalue_pro}
  a(u,v)=\lambda (u,v) \quad\forall v\in H^1_0(\Omega),
\end{eqnarray}
where $\lambda$ is the eigenvalue of the symmetric bilinear form
$a(\cdot,\cdot)$ relative to the $L^2(\Omega)$ inner product
$(\cdot,\cdot)$ and $u$ is the corresponding eigenvector.
The eigenvalues of (\ref{equ:weak_eigenvalue_pro}) are real
numbers satisfying 
$\lambda_1<\lambda_2\le\lambda_3\le\cdots$ \cite{BaOs91}.
Assume that the corresponding eigenvectors $u_1, u_2,
u_3,\dots$ satisfy $(u_i, u_j)=\delta_{ij}$ for $i,j=1,2,\cdots$
with $\delta_{ij}$ the Dirac delta.

Consider the standard Galerkin finite element
method for solving (\ref{equ:weak_eigenvalue_pro}): find a pair
$(\lambda_{\bf h}, u_{\bf h})\in \mathbb{R}\times S^{\bf h}_0(\Omega)$
satisfying
\begin{eqnarray}\label{equ:eigenvalue_FEM}
  a(u_{\bf h},v)=\lambda_{\bf h} (u_{\bf h},v) ~~~\forall v\in S^{\bf h}_0(\Omega),
\end{eqnarray}
which has finite eigenvalues
$
\lambda_{1,{\bf h}}<\lambda_{2,{\bf h}}\le \cdots \le\lambda_{n_{\bf h},{\bf h}},
$
where $n_{\bf h}=\mbox{dim} ~S^{\bf h}_0(\Omega)$. Assume that the
corresponding eigenvectors $ u_{1,{\bf h}}, u_{2,{\bf h}},\dots,
u_{n_{\bf h},{\bf h}}$ satisfy $ (u_{i,{\bf h}},$ $ u_{j,{\bf
    h}})=\delta_{ij}$. The Rayleigh principle \cite{BaOs91}
implies that $\lambda_i\le \lambda_{i,{\bf h}} \quad\text{for }i=1,2,\dots,n_{\bf h}.$

Suppose that $h=\displaystyle\max_{j\in\mathbb{Z}_d} h_j \ll 1, \lambda$ is simple, and $(\lambda_{\bf h},u_{\bf h})$  is
an approximation to $(\lambda,u)$ which satisfies \eqref{equ:eigenvalue_FEM}
and Lemma 3.2 in \cite{LSZ11} with
$\|u_{\bf h}\|_{0,\Omega}=1$. There holds \cite{BaOs91}
\begin{equation}\label{eigen-error-estimate}
  \lambda_{\bf h}-\lambda+\|u-u_{\bf h}\|_{0,\Omega}+h\|u-u_{\bf h}\|_{1,\Omega}\lesssim h^2.
\end{equation}

\begin{lem}\cite{GLZ08}\label{lemma3}
There holds
\begin{equation}\label{lemma3-eq}
\lambda_{\bf h}-\lambda=\lambda(u,u-P_{\bf h}u)+\mathcal{O}(h^{4}).
\end{equation}
\end{lem}

\section{A vector transformation operator and its implementation}\label{sec:sym-alg}

In this section, we will introduce a vector transformation operator associated with symmetric functions. A so-called symmetrization algorithm will be proposed to implement this vector transformation. We will combine the symmetrization algorithm with the two-scale finite element method to design new and efficient algorithms in Section \ref{sec:sym-ts}.

We call $u$ a symmetric function on $\bar{\Omega}$ if
\begin{equation}\label{eq:usym}
u(x_{\sigma(1)},x_{\sigma(2)},\cdots,x_{\sigma(d)})=u(x_1,x_2,\cdots,x_d), \quad \quad \forall \mathbf{x}\in \bar\Omega,~\sigma\in \text{Sym}(d).
\end{equation}
Thus we have
\begin{align}\label{eq:sym-dd-grid3}
u(i_{\sigma(1)} h_{\sigma(1)},i_{\sigma(2)} h_{\sigma(2)}, \cdots, i_{\sigma(d)} h_{\sigma(d)}) &= u(i_1 h_1,i_2 h_2, \cdots, i_d h_d), \nonumber \\
                         &\quad\quad  i_k=0,1,\cdots,N_k; k=1,2,\cdots,d.
\end{align}

Next we illustrate how to get the values of $u(\mathbf{x})$ on the grid points of $T^{h_{\sigma(1)},h_{\sigma(2)},\cdots,h_{\sigma(d)}}(\Omega)$ from those on the grid points of  $T^{h_1,h_2,\cdots,h_d}(\Omega)$ in computation. The values of $u(\mathbf{x})$ on the grid points of $T^{h_1,h_2,\cdots,h_d}(\Omega)$ are stored in a vector $U^{h_1,h_2,\cdots,h_d}$ with the size $N_{dof}:=\prod_{i=1}^d (N_i+1)$. Usually the grid points of $T^{h_1,h_2,\cdots,h_d}(\Omega)$ are numbered along the $x_1$-direction, the $x_2$-direction, $\cdots$, and the $x_d$-direction, consecutively. That is, the $\hat{I}$-th component of $U^{h_1,h_2,\cdots,h_d}$ satisfies
\begin{align}\label{eq:sym-dd-grid0}
U^{h_1,h_2,\cdots,h_d}_{\hat{I}}& = u(i_1 h_1,i_2 h_2, \cdots, i_d h_d), \nonumber \\
                         & \quad\quad i_k=0,1,\cdots,N_k; k=1,2,\cdots,d,
\end{align}
where the subscript
\begin{equation}\label{eq:sym-dd-index1}
\hat{I}:=I(i_1,i_2,\cdots,i_d) =\sum_{l=2}^d \left(i_l \prod_{j=1}^{l-1} (N_j+1)\right)+i_1+1.
\end{equation}
The values of $u(\mathbf{x})$ on the grid points of  $T^{h_{\sigma(1)},h_{\sigma(2)},\cdots,h_{\sigma(d)}}(\Omega)$ are contained in the vector $U^{h_{\sigma(1)},h_{\sigma(2)},\cdots,h_{\sigma(d)}}$. We obtain from \eqref{eq:sym-dd-grid3}-\eqref{eq:sym-dd-index1} that
\begin{align}\label{eq:sym-dd-grid4}
U^{h_{\sigma(1)},h_{\sigma(2)},\cdots,h_{\sigma(d)}}_{I(i_{\sigma(1)},i_{\sigma(2)},\cdots,i_{\sigma(d)})}& = U^{h_1,h_2,\cdots,h_d}_{I(i_1,i_2,\cdots,i_d)}, \nonumber \\
                         & \quad\quad i_k=0,1,\cdots,N_k; k=1,2,\cdots,d,
\end{align}
with
\begin{equation*}
I(i_{\sigma(1)},i_{\sigma(2)},\cdots,i_{\sigma(d)}) =\sum_{l=2}^d \left(i_{\sigma(l)} \prod_{j=1}^{l-1} (N_{\sigma(j)}+1)\right)+i_{\sigma(1)}+1.
\end{equation*}
Hence for any $\sigma\in \text{Sym}(d)$, one can get $U^{h_{\sigma(1)},h_{\sigma(2)},\cdots,h_{\sigma(d)}}$ from $U^{h_1,h_2,\cdots,h_d}$. We can define a vector transformation operator $\mathcal{T}_{\sigma}: \mathbb{C}^{N_{dof}}\rightarrow \mathbb{C}^{N_{dof}}$ satisfying
\begin{equation}\label{eq:F-def}
\mathcal{T}_{\sigma}(U^{h_1,h_2,\cdots,h_d})=U^{h_{\sigma(1)},h_{\sigma(2)},\cdots,h_{\sigma(d)}}, \quad\quad \forall U^{h_1,h_2,\cdots,h_d}\in \mathbb{C}^{N_{dof}}.
\end{equation}
This vector transformation operator can be implemented by Algorithm \ref{alg:symmetric-dd}, which is called the symmetrization algorithm.

\begin{algorithm}
\label{alg:symmetric-dd}
    \hspace*{0.02in} {\bf Input:}
       $h_1$, $h_2$, $\cdots$, $h_d$, $U^{h_1,h_2,\cdots,h_d}$, $\sigma$.\\
    \hspace*{0.02in} {\bf Output:}
       $U^{h_{\sigma(1)},h_{\sigma(2)},\cdots,h_{\sigma(d)}}$.
    \begin{algorithmic}[1]
       \State{ $N_1 = 1/h_1;\ N_2 = 1/h_2;\cdots;\ N_d = 1/h_d$}.
       \For{$i_1 = 0 : N_1$}
           \For{$i_2 = 0 : N_2$}
             \State {$\cdots$}
               \For{$i_d = 0 : N_d$}
                   \State{ $U^{h_{\sigma(1)},h_{\sigma(2)},\cdots,h_{\sigma(d)}}_{I(i_{\sigma(1)},i_{\sigma(2)},\cdots,i_{\sigma(d)})} = U^{h_1,h_2,\cdots,h_d}_{I(i_1,i_2,\cdots,i_d)}$.}
               \EndFor
           \EndFor
        \EndFor
        \State \Return $U^{h_{\sigma(1)},h_{\sigma(2)},\cdots,h_{\sigma(d)}}$.
    \end{algorithmic}
\end{algorithm}

Calculating $I(i_1,i_2,\cdots,i_d)$ in \eqref{eq:sym-dd-index1} requires $d^2$ operations. Hence the amount of operations for Algorithm \ref{alg:symmetric-dd} is $2d^2 N_{dof}$. Next we illustrate the idea of Algorithm \ref{alg:symmetric-dd} further for $d=2, 3$, respectively.

\subsection{Two dimensional case}
Let $\Omega=(0,1)^2$. We see that there is only one element $\sigma=(1,2)$ in $\text{Sym}(2)$ besides the identical operator. Assume that $u$ is a symmetric function on $\bar\Omega$. Namely,
\begin{equation}\label{eq:usym-2d}
u(x_2,x_1)=u(x_{\sigma(1)},x_{\sigma(2)})=u(x_1,x_2),\ \ \forall \mathbf{x}\in \bar\Omega.
\end{equation}
Consequently for the grid points of $T^{h_2,h_1}(\Omega)=T^{h_{\sigma(1)},h_{\sigma(2)}}(\Omega)$ and $T^{h_1, h_2}(\Omega)$, we have
\begin{align}\label{eq:sym-2d-grid1}
u(i_2 h_2,i_1 h_1)=u(i_{\sigma(1)} h_{\sigma(1)},i_{\sigma(2)} h_{\sigma(2)})&= u(i_1 h_1,i_2 h_2), \nonumber\\
& i_k=0,1,\cdots,N_k; k=1,2.
\end{align}
In implementation, the values of $u(x_1,x_2)$ on the grid points of a uniform grid $T^{h_1,h_2}(\Omega)$ are stored in a vector $U^{h_1,h_2}$ with the size $N_{dof}:=(N_1+1)\times (N_2+1)$. Usually the grid points of a uniform grid $T^{h_1,h_2}(\Omega)$ are numbered along the $x_1$-direction and the $x_2$-direction, consecutively. That is, the $\hat{I}$-th component of $U^{h_1,h_2}$ satisfies
\begin{equation}\label{eq:sym-2d-grid0}
U^{h_1,h_2}_{\hat{I}}=u(i_1 h_1,i_2 h_2),
                          \quad\quad i_k=0,1,\cdots,N_k; k=1,2,
\end{equation}
where the subscript
\begin{equation}\label{eq:sym-2d-index1}
\hat{I}:=I(i_1,i_2)={i_2(N_1+1)+i_1+1}.
\end{equation}
The values of $u(x_2,x_1)$ on the grid points of $T^{h_2,h_1}(\Omega)$ are stored in the vector $U^{h_2,h_1}$. By \eqref{eq:sym-2d-grid1} -- \eqref{eq:sym-2d-index1}, we have
\begin{equation}\label{eq:trans-2d2}
  U^{h_2,h_1}_{I(i_2,i_1)}=U^{h_1,h_2}_{I(i_1,i_2)},\quad\quad i_k=0,1,\cdots,N_k; k=1,2,
\end{equation}
with
\begin{equation*}
I(i_2,i_1)={i_1(N_2+1)+i_2+1}.
\end{equation*}
That is, one can get $U^{h_2,h_1}$ from $U^{h_1,h_2}$. Let $d=2$ in \eqref{eq:F-def}. It follows that $\mathcal{T}_{(1,2)}(U^{h_1,h_2})=U^{h_2,h_1}$.

\begin{figure}[htbp]
\centering
\includegraphics[width=10cm]{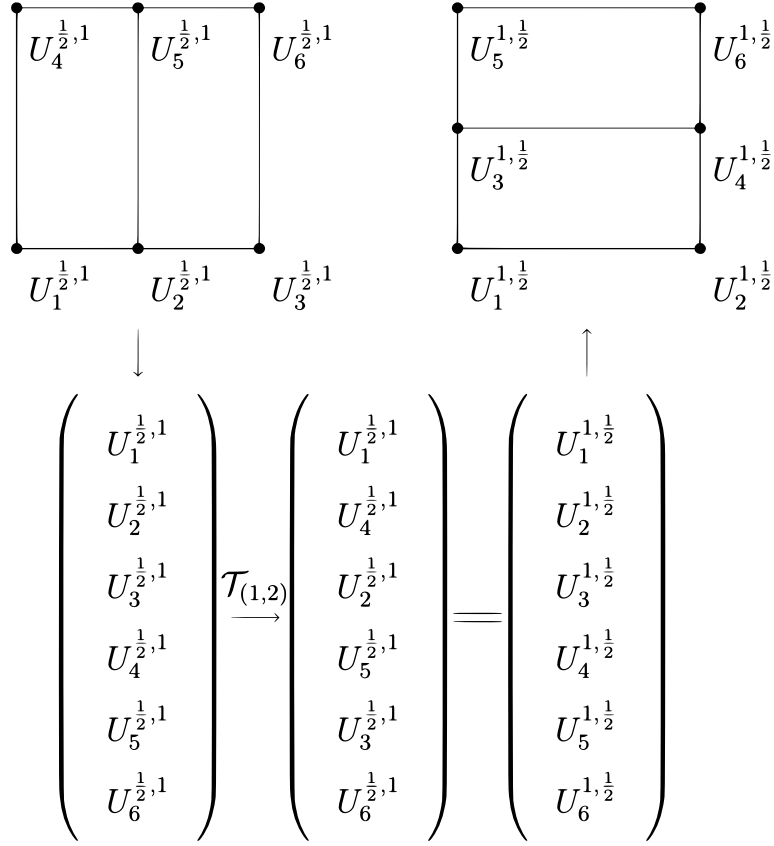}
\caption{The process of obtaining the vector $U^{1,\frac{1}{2}}$ from $U^{\frac{1}{2},1}$. For example, $U^{1,\frac{1}{2}}_{2} = U^{1,\frac{1}{2}}_{I(1,0)} = u(1,0) = u(0,1) = U^{\frac{1}{2},1}_{I(0,1)} = U_4^{\frac{1}{2},1}$ by \eqref{eq:sym-2d-grid0}, \eqref{eq:sym-2d-index1}, and \eqref{eq:trans-2d2}.}
  \label{fig:sym-2d-0}
\end{figure}

\begin{figure}[htbp]
\centering
\includegraphics[width=10cm]{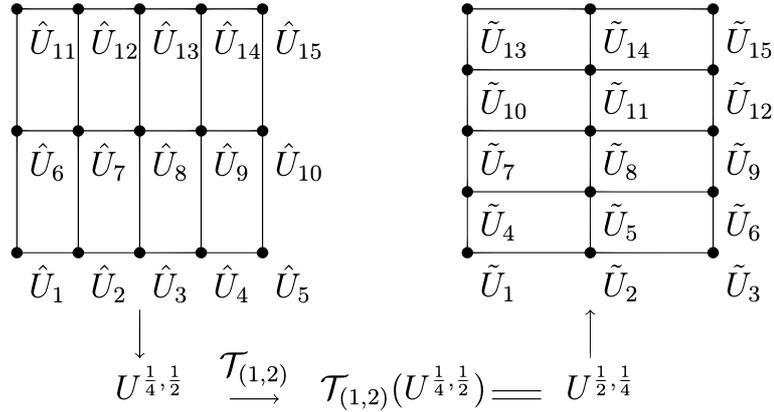}
\caption{The process of obtaining the vector $\tilde{U}:= U^{\frac{1}{2},\frac{1}{4}}$ from $\hat{U}:= U^{\frac{1}{4},\frac{1}{2}}$. For example, $U^{\frac{1}{2},\frac{1}{4}}_{11} = U^{\frac{1}{2},\frac{1}{4}}_{I(1,3)}=u(\frac{1}{2},\frac{3}{4})=u(\frac{3}{4},\frac{1}{2})= U^{\frac{1}{4},\frac{1}{2}}_{I(3,1)}= U^{\frac{1}{4},\frac{1}{2}}_9$ by \eqref{eq:sym-2d-grid0}, \eqref{eq:sym-2d-index1}, and \eqref{eq:trans-2d2}.}
  \label{fig:sym-2d-1}
\end{figure}

For example, if $h_1=\frac{1}{2}$ and $h_2=1$, it is shown in Figure \ref{fig:sym-2d-0} that the values of $u(x_1,x_2)$ on the grid points of the uniform grids $T^{\frac{1}{2},1}(\Omega)$ and $T^{1,\frac{1}{2}}(\Omega)$ are stored in the vectors $U^{\frac{1}{2},1}$ and $U^{1,\frac{1}{2}}$, respectively. By Algorithm \ref{alg:symmetric-dd}, $U^{1,\frac{1}{2}}=\mathcal{T}_{(1,2)}(U^{\frac{1}{2},1})$ can be obtained from $U^{\frac{1}{2},1}$. Similarly, if $h_1=\frac{1}{4}$ and $h_2=\frac{1}{2}$, The process of obtaining the vector $U^{\frac{1}{2},\frac{1}{4}}$ from $ U^{\frac{1}{4},\frac{1}{2}}$ is illustrated in Figure \ref{fig:sym-2d-1}, in which $\hat{U}:= U^{\frac{1}{4},\frac{1}{2}}$ and $\tilde{U}:=U^{\frac{1}{2},\frac{1}{4}}$ for simplicity.

The process of obtaining the vectors $U^{h_{\sigma(1)},h_{\sigma(2)}}=\mathcal{T}_{\sigma}(U^{h_1,h_2})$ from $U^{h_1,h_2}$  can be implemented by Algorithm \ref{alg:symmetric-dd} with $d=2$, which demands $8 N_{dof}$ operations.

\subsection{Three dimensional case}

Let $\Omega=(0,1)^3$. Assume that $u$ is a symmetric function on $\bar\Omega$. That is,
$$
u(x_{\sigma(1)},x_{\sigma(2)},x_{\sigma(3)})=u(x_1,x_2,x_3),\quad \forall \mathbf{x}\in \bar\Omega,~\sigma\in \text{Sym}(3).
$$
Thus for the grid points of $T^{h_{\sigma(1)},h_{\sigma(2)},h_{\sigma(3)}}(\Omega)$ and $T^{h_1,h_2,h_3}(\Omega)$, we have
\begin{align}\label{eq:sym-3d-grid1}
u(i_{\sigma(1)} h_{\sigma(1)},i_{\sigma(2)} h_{\sigma(2)},i_{\sigma(3)}h_{\sigma(3)})&=u(i_1 h_1,i_2 h_2,i_3 h_3), \nonumber\\
                 & i_k=0,1,\cdots,N_k; k=1,2,3.
\end{align}
In implementation, similarly, the values of $u(x_1,x_2,x_3)$ on the grid points of $T^{h_1,h_2,h_3}(\Omega)$ are stored in a vector denoted by $U^{h_1,h_2,h_3}$ with the size $N_{dof}:=\prod_{i=1}^3 (N_i+1)$. The grid points of $T^{h_1,h_2,h_3}(\Omega)$ are usually numbered along the $x_1$-direction, the $x_2$-direction, and the $x_3$-direction, consecutively. Thus the $\hat{I}$-th component of $U^{h_1,h_2,h_3}$ satisfies
\begin{align}\label{eq:sym-3d-grid0}
U^{h_1,h_2,h_3}_{\hat{I}}&=u(i_1 h_1,i_2 h_2,i_3 h_3),  \nonumber\\
& i_k=0,1,\cdots,N_k; k=1,2,3,
\end{align}
where the subscript
\begin{equation}\label{eq:sym-3d-index0}
\hat{I}:=I(i_1,i_2,i_3)=i_3(N_1+1)(N_2+1)+ i_2(N_1+1)+i_1+1.
\end{equation}
The values of $u(\mathbf{x})$ on the grid points of $T^{h_{\sigma(1)},h_{\sigma(2)},h_{\sigma(3)}}(\Omega)$ are contained in the vector $U^{h_{\sigma(1)},h_{\sigma(2)},h_{\sigma(3)}}$. By \eqref{eq:sym-3d-grid1}, \eqref{eq:sym-3d-grid0}, and \eqref{eq:sym-3d-index0}, we obtain
\begin{align}\label{eq:trans-3d2}
U^{h_{\sigma(1)},h_{\sigma(2)},h_{\sigma(3)}}_{I(i_{\sigma(1)},i_{\sigma(2)},i_{\sigma(3)})}&=U^{h_1,h_2,h_3}_{I(i_1,i_2,i_3)},\nonumber\\
& i_k=0,1,\cdots,N_k; k=1,2,3,
\end{align}
with
\begin{equation*}
I(i_{\sigma(1)},i_{\sigma(2)},i_{\sigma(3)})=i_{\sigma(3)}(N_{\sigma(1)}+1)(N_{\sigma(2)}+1)+ i_{\sigma(2)}(N_{\sigma(1)}+1)+i_{\sigma(1)}+1.
\end{equation*}
That is, the vector $U^{h_{\sigma(1)},h_{\sigma(2)},h_{\sigma(3)}}$ can be obtained from $U^{h_1,h_2,h_3}$. Let $d=3$ in \eqref{eq:F-def}. It follows that $\mathcal{T}_{\sigma}(U^{h_1,h_2,h_3})=U^{h_{\sigma(1)},h_{\sigma(2)},h_{\sigma(3)}}$.

For example, if $h_1=\frac{1}{4}$, $h_2=\frac{1}{2}$, and $h_3=\frac{1}{2}$, the components of the vector $U^{\frac{1}{4},\frac{1}{2},\frac{1}{2}}$ are illustrated in Figures \ref{fig:sym-3d-1} and \ref{fig:sym-3d-2}, where $\hat{U}:= U^{\frac{1}{4},\frac{1}{2},\frac{1}{2}}$ for simplicity.

\begin{figure}[H]
\centering
\includegraphics[width=4cm]{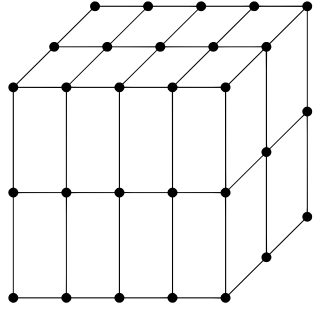}
\caption{Grid points for $U^{\frac{1}{4},\frac{1}{2},\frac{1}{2}}$.}\label{fig:sym-3d-1}
\end{figure}

\begin{figure}[H]
\centering
\includegraphics[width=13cm]{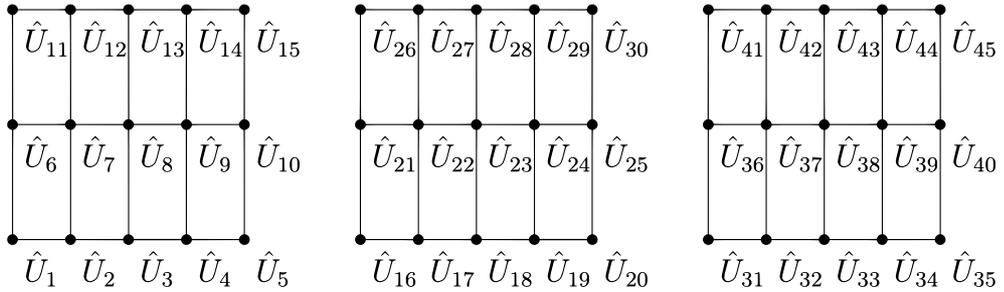}
 \caption{Components of $\hat{U}:= U^{\frac{1}{4},\frac{1}{2},\frac{1}{2}}$ on each layer along the $x_3$-direction. For example, $\hat{U}_{24}:= U^{\frac{1}{4},\frac{1}{2},\frac{1}{2}}_{24} = U^{\frac{1}{4},\frac{1}{2},\frac{1}{2}}_{I(3,1,1)}=u(\frac{3}{4},\frac{1}{2},\frac{1}{2})$ by \eqref{eq:sym-3d-grid0} and \eqref{eq:sym-3d-index0}.}\label{fig:sym-3d-2}
\end{figure}

The vectors $U^{\frac{1}{2},\frac{1}{4},\frac{1}{2}}=\mathcal{T}_{(1,2)}(U^{\frac{1}{4},\frac{1}{2},\frac{1}{2}})$ and $U^{\frac{1}{2},\frac{1}{2},\frac{1}{4}}=\mathcal{T}_{(1,3)}(U^{\frac{1}{4},\frac{1}{2},\frac{1}{2}})$ can be obtained from $U^{\frac{1}{4},\frac{1}{2},\frac{1}{2}}$ with $\sigma=(1,2)$ and $(1,3)$ in \eqref{eq:trans-3d2}, respectively. It is illustrated in Figures \ref{fig:sym-3d-3} and \ref{fig:sym-3d-4}.

\begin{figure}[H]
\centering
\includegraphics[width=13cm]{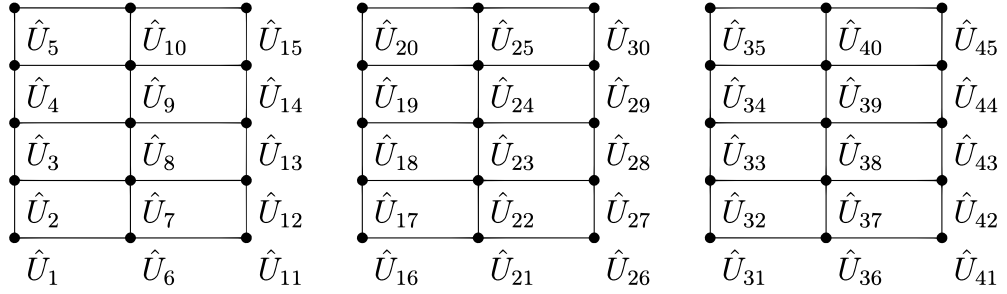}
  \caption{On each layer along the $x_3$-direction, we obtain $U^{\frac{1}{2},\frac{1}{4},\frac{1}{2}}$ from $\hat{U}:=U^{\frac{1}{4},\frac{1}{2},\frac{1}{2}}$ by \eqref{eq:trans-3d2}. For example, $U^{\frac{1}{2},\frac{1}{4},\frac{1}{2}}_{I(0,3,1)}= u(0,\frac{3}{4},\frac{1}{2})=u(\frac{3}{4},0,\frac{1}{2})=U^{\frac{1}{4},\frac{1}{2},\frac{1}{2}}_{I(3,0,1)}$. That is, $U^{\frac{1}{2},\frac{1}{4},\frac{1}{2}}_{25}=U^{\frac{1}{4},\frac{1}{2},\frac{1}{2}}_{19}=\hat{U}_{19}$. }\label{fig:sym-3d-3}
\end{figure}

\begin{figure}[H]
\centering
\includegraphics[width=13cm]{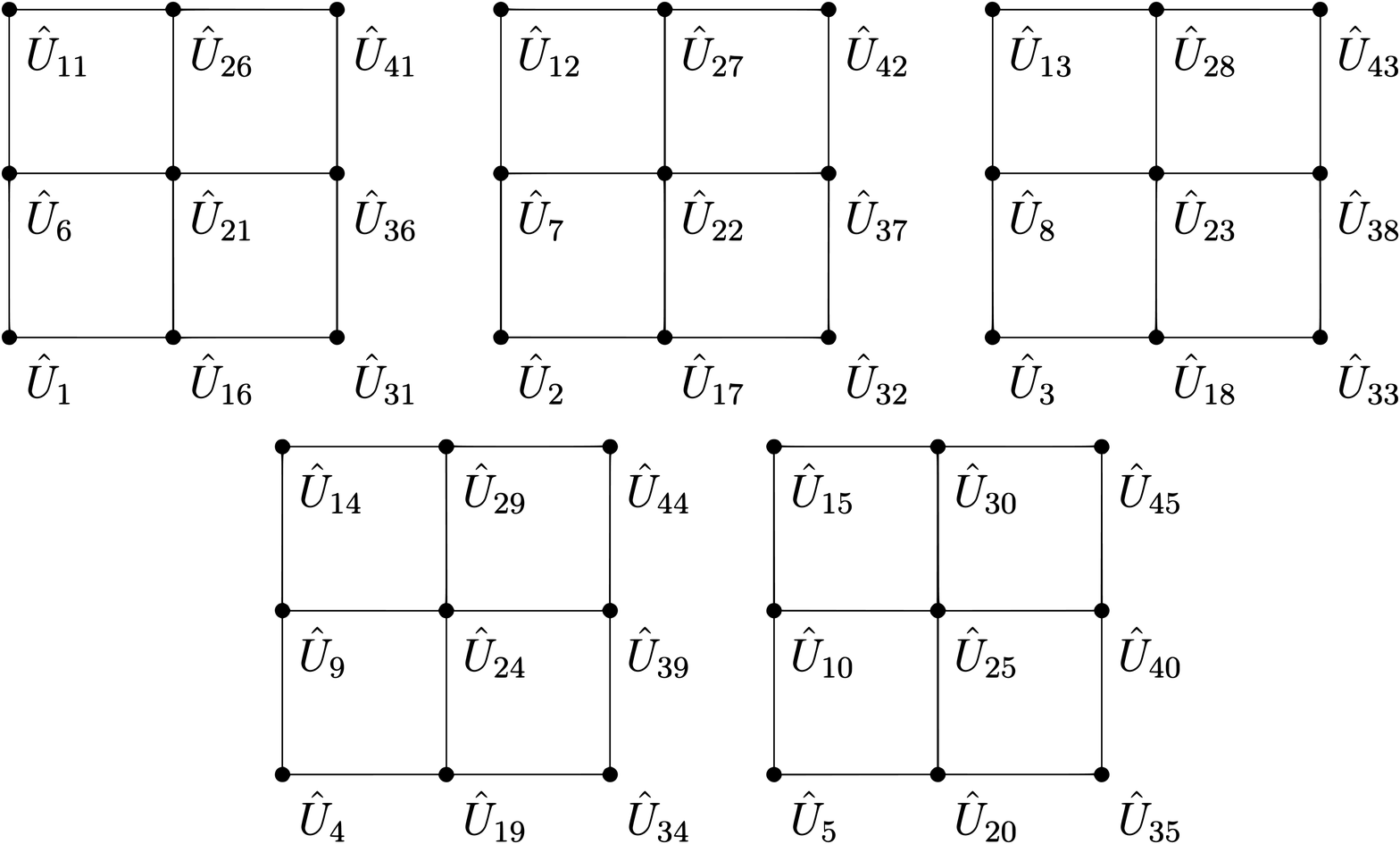}
  \caption{On each layer along the $x_3$-direction, we get $U^{\frac{1}{2},\frac{1}{2},\frac{1}{4}}$ from $\hat{U}:= U^{\frac{1}{4},\frac{1}{2},\frac{1}{2}}$ by \eqref{eq:trans-3d2}.  For example, $U^{\frac{1}{2},\frac{1}{2},\frac{1}{4}}_{I(0,2,1)}=u(0,1,\frac{1}{4})=u(\frac{1}{4},1,0)=U^{\frac{1}{4},\frac{1}{2},\frac{1}{2}}_{I(1,2,0)}$. That is, $U^{\frac{1}{2},\frac{1}{2},\frac{1}{4}}_{16}=U^{\frac{1}{4},\frac{1}{2},\frac{1}{2}}_{12}=\hat{U}_{12}$. }\label{fig:sym-3d-4}
\end{figure}

The process of obtaining the vectors $U^{h_2,h_1,h_3}=U^{h_{(1,2)(1)},h_{(1,2)(2)},h_{(1,2)(3)}}$ or $U^{h_3,h_2,h_1}=U^{h_{(1,3)(1)},h_{(1,3)(2)},h_{(1,3)(3)}}$ from $U^{h_1,h_2,h_3}$ can be implemented by Algorithm \ref{alg:symmetric-dd} with $d=3$, which requires $18 N_{dof}$ operations.

\section{The symmetrized two-scale finite element approximations}\label{sec:sym-ts}

For the source problem \eqref{equ:weak_form_boundary_value_pro} and eigenvalue problem \eqref{equ:weak_eigenvalue_pro}, the two-scale finite element method has been developed \cite{GLZ08,LZ06}. In this section, for  \eqref{equ:weak_form_boundary_value_pro} and \eqref{equ:weak_eigenvalue_pro} with symmetric solutions, Algorithm \ref{alg:symmetric-dd} will be combined with the two-scale finite element method.

Let $v$ be a symmetric function and $v_{h_1,h_2,\cdots,h_d}\in {S}^{h_1,h_2,\cdots,h_d}(\Omega)$ be a standard finite element approximation to $v$. As stated in Section \ref{sec:sym-alg}, usually the grid points of $T^{h_1,h_2,\cdots,h_d}(\Omega)$ are numbered along the $x_1$-direction, the $x_2$-direction, $\cdots$, and the $x_d$-direction, consecutively. The values of $v_{h_1,h_2,\cdots,h_d}$ on these nodes are stored in a vector denoted by $V^{h_1,h_2,\cdots,h_d}$ with the size $N_{dof}:=\prod_{i=1}^d (N_i+1)$.

Define a bijection $G: \mathbb{C}^{N_{dof}}\rightarrow S^{h_1,h_2,\cdots,h_d}(\Omega)$ satisfying
\begin{equation*}
G(V^{h_1,h_2,\cdots,h_d})=\sum_{i=1}^{N_{dof}} V^{h_1,h_2,\cdots,h_d}_i \phi_i, \quad \quad \forall  V^{h_1,h_2,\cdots,h_d}\in \mathbb{C}^{N_{dof}},
\end{equation*}
where $V^{h_1,h_2,\cdots,h_d}_i$ is the $i$-th component of $V^{h_1,h_2,\cdots,h_d}$ and $\phi_i$ is the Lagrangian basis of $S^{h_1,h_2,\cdots,h_d}(\Omega)$ with $i=1,2,\cdots,N_{dof}$. Obviously we have $v_{h_1,h_2,\cdots,h_d}=G(V^{h_1,h_2,\cdots,h_d})$. For any $\sigma \in \text{Sym}(d)$, there holds
$$
v_{h_{\sigma(1)},h_{\sigma(2)},\cdots, h_{\sigma(d)}}=G(V^{h_{\sigma(1)},h_{\sigma(2)},\cdots, h_{\sigma(d)}}).
$$
Invoking \eqref{eq:F-def} we obtain
\begin{equation*}
\mathcal{T}_{\sigma}(V^{h_1,h_2,\cdots,h_d})=V^{h_{\sigma(1)},h_{\sigma(2)},\cdots,h_{\sigma(d)}},\quad \forall V^{h_1,h_2,\cdots,h_d}\in  \mathbb{C}^{N_{dof}}.
\end{equation*}
The process of obtaining $V^{h_{\sigma(1)},h_{\sigma(2)},\cdots, h_{\sigma(d)}}$ from $V^{h_1,h_2,\cdots,h_d}$ can be implemented by a symmetrization algorithm (Algorithm \ref{alg:symmetric-dd}). Define a bijection $R_{\sigma}: S^{h_1,h_2,\cdots,h_d}(\Omega)\rightarrow S^{h_{\sigma(1)},h_{\sigma(2)},\cdots,h_{\sigma(d)}}(\Omega)$ satisfying
\begin{equation}\label{eq:R-def-d}
R_{\sigma}=G\circ \mathcal{T}_{\sigma} \circ G^{-1}.
\end{equation}
We have
\begin{equation}\label{eq:trans-intro}
v_{h_{\sigma(1)},h_{\sigma(2)},\cdots,h_{\sigma(d)}} =R_{\sigma}(v_{h_1,h_2,\cdots,h_d}),\quad \forall v_{h_1,h_2,\cdots,h_d}\in S^{h_1,h_2,\cdots,h_d}(\Omega).
\end{equation}

It is seen from \eqref{eq:R-def-d} and \eqref{eq:trans-intro} that the critical process of getting $R_{\sigma}(v_{h_1,h_2,\cdots,h_d})$ from $v_{h_1,h_2,\cdots,h_d}$ is to calculate $\mathcal{T}_{\sigma}(V^{h_1,h_2,\cdots,h_d})$ from $V^{h_1,h_2,\cdots,h_d}$ by Algorithm \ref{alg:symmetric-dd}, which requires $2d^2 N_{dof}$ operations with $N_{dof}:=\prod_{i=1}^d (N_i+1)$. While computing $v_{h_{\sigma(1)},h_{\sigma(2)},\cdots,h_{\sigma(d)}}$ by the standard finite element method usually demands $M N_{dof}$ operations with an amount $M \gg 2d^2$, due to creating stiffness matrix and solving linear algebraic systems. Consequently, if $v_{h_1,h_2,\cdots,h_d}$ has been calculated, then obtaining $v_{h_{\sigma(1)},h_{\sigma(2)},\cdots,h_{\sigma(d)}}$ through Algorithm \ref{alg:symmetric-dd} is much more efficient than the standard finite element method. From this observation, some efficient algorithms for \eqref{equ:weak_form_boundary_value_pro} and \eqref{equ:weak_eigenvalue_pro} with symmetric solutions can be established.

\subsection{The source problem}\label{sec:bvp}

For the source problem \eqref{equ:weak_form_boundary_value_pro} with symmetric solution, we propose a symmetrized two-scale finite element method to reduce computational cost by combining Algorithm \ref{alg:symmetric-dd} with the two-scale finite element method in \cite{GLZ08,LZ06}. Let $h,H\in(0,1)$ and assume that $H/h$ is a positive integer. Let $w_{h\alpha+H(\textbf{e}-\alpha)} \in S_0^{h\alpha+H(\textbf{e}-\alpha)}(\Omega)(\textbf{0}\le\alpha\le\textbf{e})$, and define
\begin{equation*}
    B^h_{H\textbf{e}}w_{h\textbf{e}}=\sum_{i=1}^d w_{H\hat{\textbf{e}}_i+h\textbf{e}_i}-(d-1)w_{H\textbf{e}}.
\end{equation*}
For instance, $B^h_{H,H,H}w_{h,h,h}=w_{h,H,H}+w_{H,h,H}+w_{H,H,h}-2w_{H,H,H}$, for $d=3$. Invoking \eqref{eq:trans-intro}, a symmetrized two-scale finite element method for the linear source problem \eqref{equ:weak_form_boundary_value_pro} with symmetric solution is described in Algorithm \ref{alg:bvp_ts}.

\begin{algorithm}\label{alg:bvp_ts}
    \begin{enumerate}
        \item Solve (\ref{equ:boundary_FEM}) on a coarse grid: find $P_{H\textbf{e}}u\in S_0^{      H\textbf{e}}(\Omega)$\ such that
        \begin{equation*}
            a(P_{H\textbf{e}}u,v)=(f,v)\ \ \ \ \forall v\in S_0^{H\textbf{e}}(\Omega).
        \end{equation*}
        \item Solve (\ref{equ:boundary_FEM}) on a grid which is fine in the first coordinate direction: find $P_{h\textbf{e}_1+H\hat{\textbf{e}}_1}u\in S_0^{h\textbf{e}_1+H\hat{\textbf{e}}_1}(\Omega)$\ such that
        \begin{align*}
            a(P_{h\textbf{e}_1+H\hat{\textbf{e}}_1}u,v)=(f,v)\ \ \ \ \forall v\in S_0^{h\textbf{e}_1+H\hat{\textbf{e}}_1}(\Omega) .
        \end{align*}

        Get $R_{(1,i)}(P_{h\textbf{e}_1+H\hat{\textbf{e}}_1}u) \in S_0^{h\textbf{e}_i+H\hat{\textbf{e}}_i}(\Omega), \ i\in\mathbb{Z}_d/\{1\}$ from $P_{h\textbf{e}_1+H\hat{\textbf{e}}_1}u$ through Algorithm \ref{alg:symmetric-dd}.

        \item Set
        \begin{align*}
            B^h_{H\textbf{e}}P_{h\textbf{e}}u=P_{h\textbf{e}_1+H\hat{\textbf{e}}_1}u+\sum_{i=2}^d R_{(1,i)}(P_{h\textbf{e}_1+H\hat{\textbf{e}}_1}u)-(d-1)P_{H\textbf{e}}u.
        \end{align*}
    \end{enumerate}
\end{algorithm}

For $d=2$ or $3$, there holds the following error estimates for the symmetrized two-scale finite element solution.

\begin{thm}\label{thm:bvp_ts}
    Assume that $\partial_{x_l} a_{ij}\in W^{1,\infty}(\Omega)$ and $\partial_{x_l} b_l\in L^{\infty}(\Omega)$ for $i,j,l\in \mathbb{Z}_d$. If $u\in H_0^1(\Omega)\cap W^{G,4}(\Omega)$ and $B^h_{H\textbf{e}}P_{h\textbf{e}}u$ is obtained from Algorithm \ref{alg:bvp_ts}, then
    \begin{align*}
        \lVert u-B^h_{H\textbf{e}}P_{h\textbf{e}}u \rVert_{1,\Omega}&\lesssim (h+H^3)\lVert u\rVert_{W^{G,4}(\Omega)},\\
        \lVert u-B^h_{H\textbf{e}}P_{h\textbf{e}}u \rVert_{0,\Omega}&\lesssim (h^2+H^4)\lVert u\rVert_{W^{G,4}(\Omega)}.
    \end{align*}
\end{thm}
\begin{proof}
Invoking Theorem 3.3 and Lemma 4.1 in \cite{LZ16}, we have
\begin{equation}\label{eq:thm-bvp_ts-1}
\lVert P_{h\textbf{e}}u-B^h_{H\textbf{e}}P_{h\textbf{e}}u \rVert_{1,\Omega}\lesssim (h^2+H^3)\lVert u\rVert_{W^{G,4}(\Omega)}.
\end{equation}
Then imitating the proof of Theorem 2.5 in \cite{LSZ11}, we complete the proof.
\end{proof}

\begin{remark}
We may see from Algorithm \ref{alg:bvp_ts} that the symmetrized two-scale finite element method is quite applicable for the source problem \eqref{equ:boundary_value_pro} with symmetric solution. $R_{(1,i)}(P_{h\textbf{e}_1+H\hat{\textbf{e}}_1}u)$ can be obtained from $P_{h\textbf{e}_1+H\hat{\textbf{e}}_1}u$ through Algorithm \ref{alg:symmetric-dd}. The computational cost of Algorithm \ref{alg:symmetric-dd} is far smaller than computing $P_{h\textbf{e}_i+H\hat{\textbf{e}}_i}u$ by the standard finite element method. For example, when $d=3$, the amount of operations for Algorithm \ref{alg:symmetric-dd} to get $P_{H,h,H}u=R_{(1,2)}(P_{h,H,H}u)$ from $P_{h,H,H}u$ is only $18 N_{dof}$ with $N_{dof}:=(n-1)(N-1)(N-1)$. While computing $P_{H,h,H}u$ by the standard finite element method usually requires $M N_{dof}$ operations with $M\gg 18$, due to creating stiffness matrix and solving linear algebraic systems. Hence the computational cost is reduced significantly.
\end{remark}

\begin{remark}
Combining the two-scale finite element method with the postprocessing technique, the postprocessed two-scale finite element method and the two-scale postprocessed finite element method have been proposed in \cite{LSZ11,LZ16}, in which the interpolation postprocessing operator $\Pi_{2\textbf{h}}$ for uniform tensor product grids $T^{\bf h}(\Omega)$ is used. For the linear source problem \eqref{equ:boundary_value_pro} with symmetric solution, these methods can also be combined with Algorithm \ref{alg:symmetric-dd} similarly to reduce computational cost further.
\end{remark}

\subsection{The eigenvalue problem}\label{sec:eig}

Similar to Section \ref{sec:bvp}, for the eigenvalue problem \eqref{equ:weak_eigenvalue_pro} with symmetric solution, we combine Algorithm \ref{alg:symmetric-dd} with the two-scale finite element method in \cite{GLZ08,LZ06} to propose the symmetrized two-scale finite element method (Algorithm \ref{alg:eig_ts}).

\begin{algorithm}\label{alg:eig_ts}
    \begin{enumerate}
        \item Solve (\ref{equ:eigenvalue_FEM}) on a coarse grid: find $(\lambda_{H\textbf{e}},u_{H\textbf{e}})\in \mathbb{R}\times S_0^{H\textbf{e}}(\Omega)$\ satisfying $\lVert u_{H\textbf{e}} \rVert_{0,\Omega}=1$ and
        \begin{equation*}
            a(u_{H\textbf{e}},v)=\lambda_{H\textbf{e}}(u_{H\textbf{e}},v)\ \ \ \ \forall v\in S_0^{H\textbf{e}}(\Omega).
        \end{equation*}

        \item Solve (\ref{equ:eigenvalue_FEM}) on a grid which is fine in the first coordinate direction:

        find $(\lambda_{h\textbf{e}_1+H\hat{\textbf{e}}_1}, u_{h\textbf{e}_1+H\hat{\textbf{e}}_1})\in \mathbb{R}\times S_0^{h\textbf{e}_1+H\hat{\textbf{e}}_1}(\Omega)$\ satisfying $\lVert u_{h\textbf{e}_1+H\hat{\textbf{e}}_1} \rVert_{0,\Omega}=1$ and
        \begin{align*}
            a(u_{h\textbf{e}_1+H\hat{\textbf{e}}_1},v)=\lambda_{h\textbf{e}_1+H\hat{\textbf{e}}_1}(u_{h\textbf{e}_1+H\hat{\textbf{e}}_1},v)\ \ \ \ \forall v\in S_0^{h\textbf{e}_1+H\hat{\textbf{e}}_1}(\Omega).
        \end{align*}
        For $i\in\mathbb{Z}_d/\{1\}$, set $\lambda_{h\textbf{e}_i+H\hat{\textbf{e}}_i}=\lambda_{h\textbf{e}_1+H\hat{\textbf{e}}_1}$ and obtain $ R_{(1,i)}(u_{h\textbf{e}_1+H\hat{\textbf{e}}_1})\in S_0^{h\textbf{e}_i+H\hat{\textbf{e}}_i}(\Omega)$ from  $u_{h\textbf{e}_1+H\hat{\textbf{e}}_1}\in S_0^{h\textbf{e}_1+H\hat{\textbf{e}}_1}(\Omega)$ through Algorithm \ref{alg:symmetric-dd}.

        \item Set
        \begin{align*}
            B^h_{H\textbf{e}}\lambda_{h\textbf{e}}&=\sum_{i=1}^d \lambda_{h\textbf{e}_i+H\hat{\textbf{e}}_i}-(d-1)\lambda_{H\textbf{e}},\\
            B^h_{H\textbf{e}}u_{h\textbf{e}}&=u_{h\textbf{e}_1+H\hat{\textbf{e}}_1}+\sum_{i=2}^d R_{(1,i)}(u_{h\textbf{e}_1+H\hat{\textbf{e}}_1})-(d-1)u_{H\textbf{e}}.
        \end{align*}
    \end{enumerate}
\end{algorithm}

From \eqref{eq:trans-intro} we have $u_{h\textbf{e}_i+H\hat{\textbf{e}}_i}=R_{(1,i)}(u_{h\textbf{e}_1+H\hat{\textbf{e}}_1})$. For $d=2$ or $3$, there holds the following result, which is a refinement of the results in \cite{GLZ08,LZ06}.

\begin{thm}\label{thm:eig_ts}
   Let $(\lambda_{H\textbf{e}},u_{H\textbf{e}})$ and $(\lambda_{h\textbf{e}_1+H\hat{\textbf{e}}_1}, u_{h\textbf{e}_1+H\hat{\textbf{e}}_1})$ be eigenpairs of \eqref{equ:eigenvalue_FEM} that approximate the same exact eigenpair $(\lambda,u)$. If $\partial_{x_l} a_{ij}\in W^{1,\infty}(\Omega)$ for $i,j,l\in \mathbb{Z}_d$ and $u\in H_0^1(\Omega)\cap W^{G,4}(\Omega)$, then
   \begin{align}
       \lVert u-B^h_{H\textbf{e}}u_{h\textbf{e}} \rVert_{1,\Omega} &\lesssim h+H^3, \label{thm:eig_ts-1}\\
       \lvert \lambda-B^h_{H\textbf{e}}\lambda_{h\textbf{e}} \rvert &\lesssim h^2+H^4. \label{thm:eig_ts-2}
   \end{align}
\end{thm}
\begin{proof}
Combining Lemma 5.2 with Theorem 3.3 in \cite{LZ16}, we obtain
\begin{equation}\label{eq:eig_ts-1}
\lVert B^h_{H\textbf{e}}u_{h\textbf{e}} - u_{h\textbf{e}}\rVert_{1,\Omega} \lesssim h^2+H^3.
\end{equation}
Thus we arrive at \eqref{thm:eig_ts-1}. By Lemma \ref{lemma3}, we have
\begin{eqnarray*}
& &B^h_{H{\bf e}}\lambda_{h{\bf e}}-\lambda_{h{\bf e}}=\lambda(u,
P_{h{\bf e}}u-B^h_{H{\bf e}}P_{h{\bf e}}u)+\mathcal{O}(H^{4}).
\end{eqnarray*}
Applying Theorem \ref{thm:bvp_ts}, we complete the proof.
\end{proof}

\begin{remark}
For the eigenvalue problem with symmetric eigenfunction, Algorithm \ref{alg:symmetric-dd} is combined with the two-scale finite element method to obtain new and more efficient algorithm (Algorithm \ref{alg:eig_ts}). Algorithm \ref{alg:symmetric-dd} can also be combined with the postprocessed two-scale finite element method in \cite{LSZ11} and the two-scale postprocessed finite element method in \cite{LZ16}. Besides, some two-scale finite element discretizations for the nonlinear eigenvalue problems have been proposed and analyzed in \cite{hou2021}. Consequently, for the nonlinear eigenvalue problems with symmetric eigenfunctions, the two-scale finite element methods in \cite{hou2021} can also be combined with Algorithm \ref{alg:symmetric-dd} to get new algorithms.
\end{remark}

\section{Numerical examples}\label{sec:numer-ex}
In this section, several numerical examples, including the electronic structure calculations, are presented to illustrate the efficiency of our approaches. To optimize the cost of the computation, we choose $h=H^2$ approximately.

\textbf{Example 1} Consider a source problem with a singular coefficient:
\begin{eqnarray*}
\vspace{-0.8cm}
    \left\{\begin{aligned}
     -\Delta u - \frac{1}{\sqrt{x_1^2+x_2^2+x_3^2}}u + x_1x_2x_3u &= f \ \ in\ \ \Omega=(0,1)^3,\\
     u &= 0 \ \ on\ \  \partial\Omega
    \end{aligned}
    \right.
    \vspace{-0.8cm}
\end{eqnarray*}
with $f$ chosen so that the exact solution is
$$u=2x_1x_2x_3(1-x_1)(1-x_2)(1-x_3)\text{e}^{x_1+x_2+x_3}.$$ Let $L=1.0$ be the side length of $\Omega$.

For this linear source problem with symmetric solution, the numerical results are presented in Figures \ref{fig:ex1-1}-\ref{fig:ex1-7} and Tables \ref{ex1-3}-\ref{ex1-4}. Figure \ref{fig:ex1-1} supports the convergence results of Theorem \ref{thm:bvp_ts}. Time consumptions of different methods are presented in Figure \ref{fig:ex1-7}. It is shown that the two-scale finite element method is more efficient than the standard finite element method. Algorithm \ref{alg:bvp_ts} reduces the time consumption further compared with the two-scale finite element method, which is illustrated in more details in Tables \ref{ex1-3}-\ref{ex1-4}. All of these results show that for this linear source problem with symmetric solution, considering both accuracy and computational cost, the symmetrized two-scale finite element method (Algorithm \ref{alg:bvp_ts}) is the preferred method.

\begin{figure}[H]
\vspace{-0.2cm}
\begin{minipage}[t]{0.5\linewidth}
\centering
\includegraphics[width=8cm]{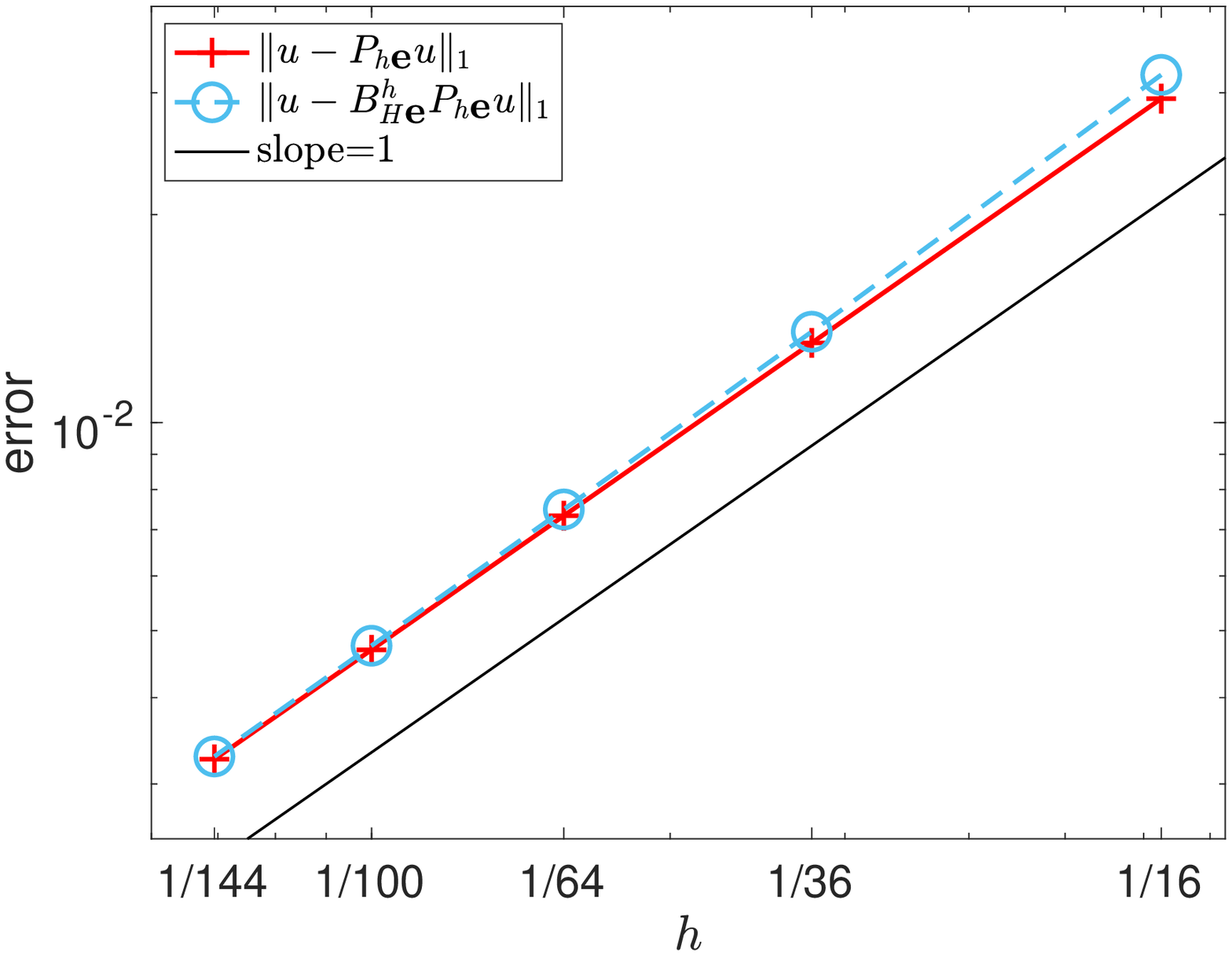}
\end{minipage}
\hfill
\begin{minipage}[t]{0.5\linewidth}
\centering
\includegraphics[width=8cm]{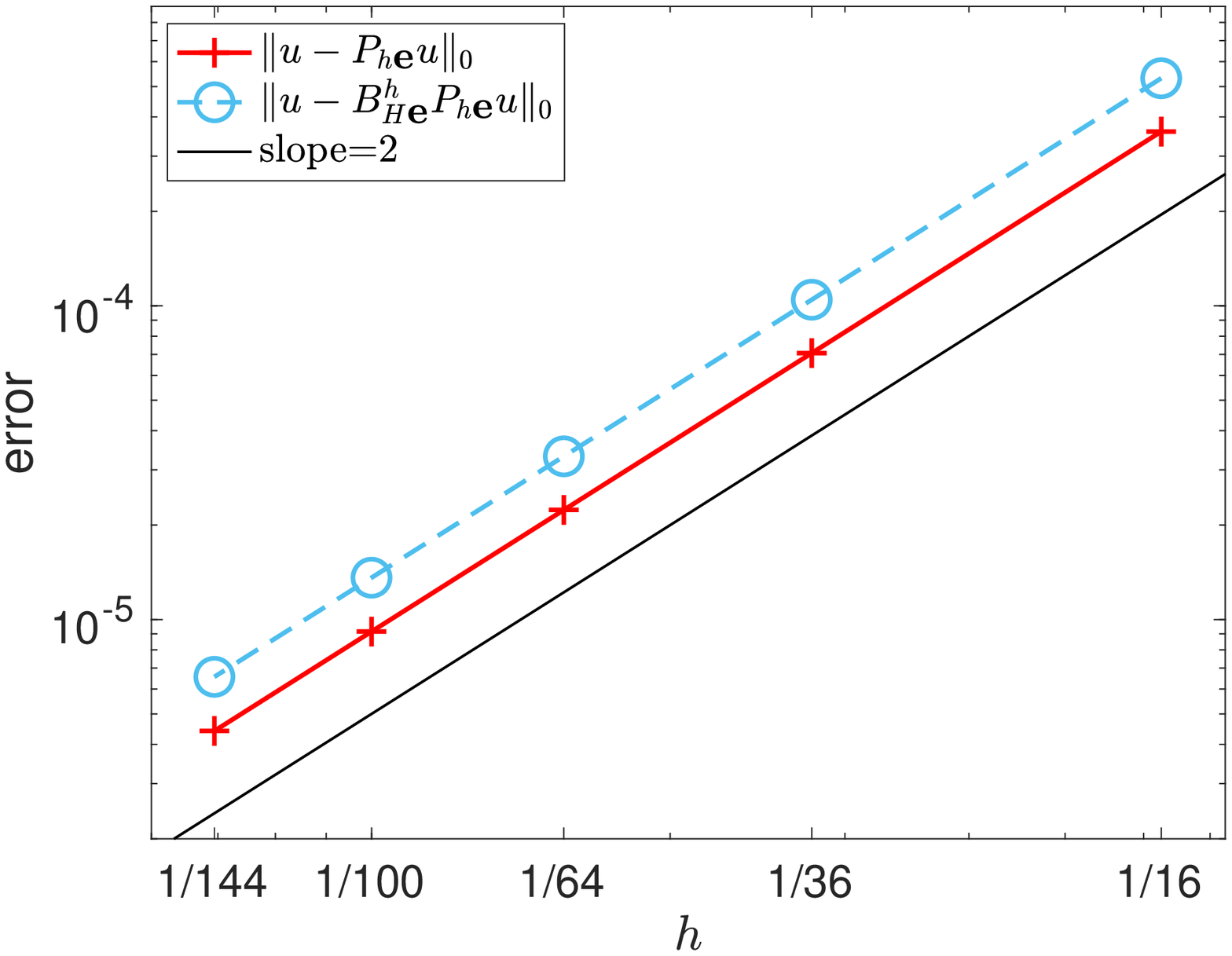}
\end{minipage}
\caption{\footnotesize(Example 1) The convergence curves of solutions in $H^1$ norm and $L^2$ norm by the standard FEM and Algorithm {\ref{alg:bvp_ts}}.}\label{fig:ex1-1}
\end{figure}

\begin{figure}[H]
\vspace{-0.2cm}
\centering
\includegraphics[width=8cm]{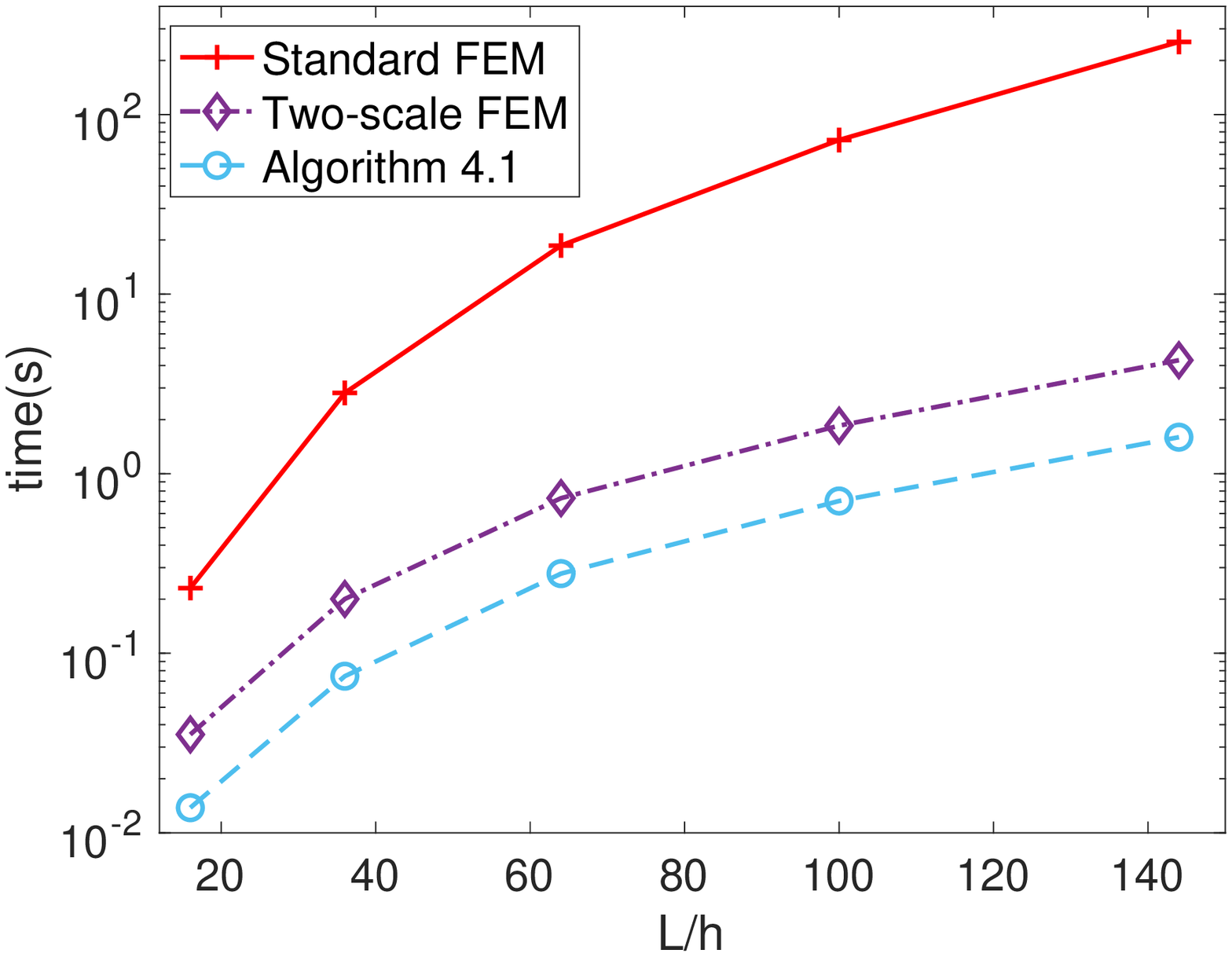}
\caption{\footnotesize(Example 1) The time consumptions of the three methods.}\label{fig:ex1-7}
\end{figure}

\begin{table}[H]
\caption{\footnotesize Example 1: The running time (s) by the two-scale FEM.}\label{ex1-3}
  \centering
    \begin{tabular}{|c|c|c|c|c|c|c|}
        \hline
         $\frac{L}{h}\times \frac{L}{H}\times \frac{L}{H}$ & $P_{H\textbf{e}}u$ & $P_{h\textbf{e}_1+H\hat{\textbf{e}}_1}u$ & $P_{h\textbf{e}_2+H\hat{\textbf{e}}_2}u$ & $P_{h\textbf{e}_3+H\hat{\textbf{e}}_3}u$ & \begin{minipage}{1.6cm}\vspace{0mm} \center $B_{H\textbf{e}}^hP_{h\textbf{e}}u$ \vspace{-2mm} \end{minipage}  & Total \\
        \hline
        $16\times4\times4$    & 0.0030 & 0.0104 & 0.0110 & 0.0104 & 0.0006 & 0.0354 \\
        \hline
        $36\times6\times6$   & 0.0106 & 0.0609 & 0.0624 & 0.0614 &  0.0052 & 0.2005 \\
        \hline
        $64\times8\times8$   & 0.0364 & 0.2250 & 0.2220 & 0.2192 &  0.0281 & 0.7307 \\
        \hline
        $100\times10\times10$ & 0.0754 & 0.5574 & 0.5484 & 0.5596 &  0.1104 & 1.8512 \\
        \hline
        $144\times12\times12$ & 0.1005 & 1.2790 & 1.2835 & 1.2941 &  0.3264 & 4.2835 \\
        \hline
    \end{tabular}
\end{table}

\begin{table}[H]
\caption{\footnotesize Example 1: The running time (s) by Algorithm \ref{alg:bvp_ts}.}\label{ex1-4}
  \centering
    \begin{tabular}{|c|c|c|c|c|c|}
        \hline
          $\frac{L}{h}\times \frac{L}{H}\times \frac{L}{H}$ & $P_{H\textbf{e}}u$ & $P_{h\textbf{e}_1+H\hat{\textbf{e}}_1}u$ & $P_{h\textbf{e}_2+H\hat{\textbf{e}}_2}u, P_{h\textbf{e}_3+H\hat{\textbf{e}}_3}u$ & \begin{minipage}{1.6cm}\vspace{0mm} \center $B_{H\textbf{e}}^hP_{h\textbf{e}}u$ \vspace{-2mm} \end{minipage} & Total \\
        \hline
        $16\times4\times4$    & 0.0030 & 0.0104 & 0.0001 & 0.0003 & 0.0138 \\
        \hline
        $36\times6\times6$   & 0.0106 & 0.0609 & 0.0004 &  0.0026 & 0.0745 \\
        \hline
        $64\times8\times8$   & 0.0364 & 0.2250 & 0.0027 &  0.0135 & 0.2776 \\
        \hline
        $100\times10\times10$ & 0.0754 & 0.5574 & 0.0136 &  0.0581 & 0.7045 \\
        \hline
        $144\times12\times12$ & 0.1005 & 1.2790 & 0.0461 &  0.1708 & 1.5964 \\
        \hline
    \end{tabular}
\end{table}

\textbf{Example 2} Consider a source problem:
\begin{eqnarray*}
    \left\{\begin{aligned}
     -\sum\limits_{i,j=1}^3\frac{\partial}{\partial x_j}(a_{ij}\frac{\partial u}{\partial x_i})+\sum\limits_{i=1}^3 b_i\frac{\partial u}{\partial x_i}+u &= f \ \ in\ \ \Omega=(1,2)^3,\\
      u &= 0 \ \ on\ \   \partial\Omega,
    \end{aligned}\right.
\end{eqnarray*}
where
\begin{equation*}
    A=(a_{ij})=
    \begin{pmatrix}
          \text{e}^{x_1} & 1           & 1 \\
          1          & \text{e}^{x_2}  & 1 \\
          1          & 1           & \text{e}^{x_3}
    \end{pmatrix},
    b_1 = b_2 = b_3 =0.001,
\end{equation*}
and $f$ is chosen so that we have the symmetric exact solution as follows
\begin{equation*}
    u=(1-x_1)(2-x_1)(1-x_2)(2-x_2)(1-x_3)(2-x_3)(\sin\sqrt{x_1x_2x_3})\text{e}^{x_1+x_2+x_3}.
\end{equation*}
Let $L=1.0$ be the side length of $\Omega$.

The numerical results are presented in Figures \ref{fig:ex3-1}-\ref{fig:ex3-7} and Tables \ref{ex3-3}-\ref{ex3-4} for the nonsymmetric source problem with symmetric solution. Similar to Example 1, Figure \ref{fig:ex3-1} also supports the convergence results of Theorem \ref{thm:bvp_ts}. It is shown that the symmetrized two-scale finite element method (Algorithm \ref{alg:bvp_ts}) is still better than the other methods.

\begin{figure}[H]
\vspace{-0.2cm}
\begin{minipage}[t]{0.5\linewidth}
\centering
\includegraphics[width=8cm]{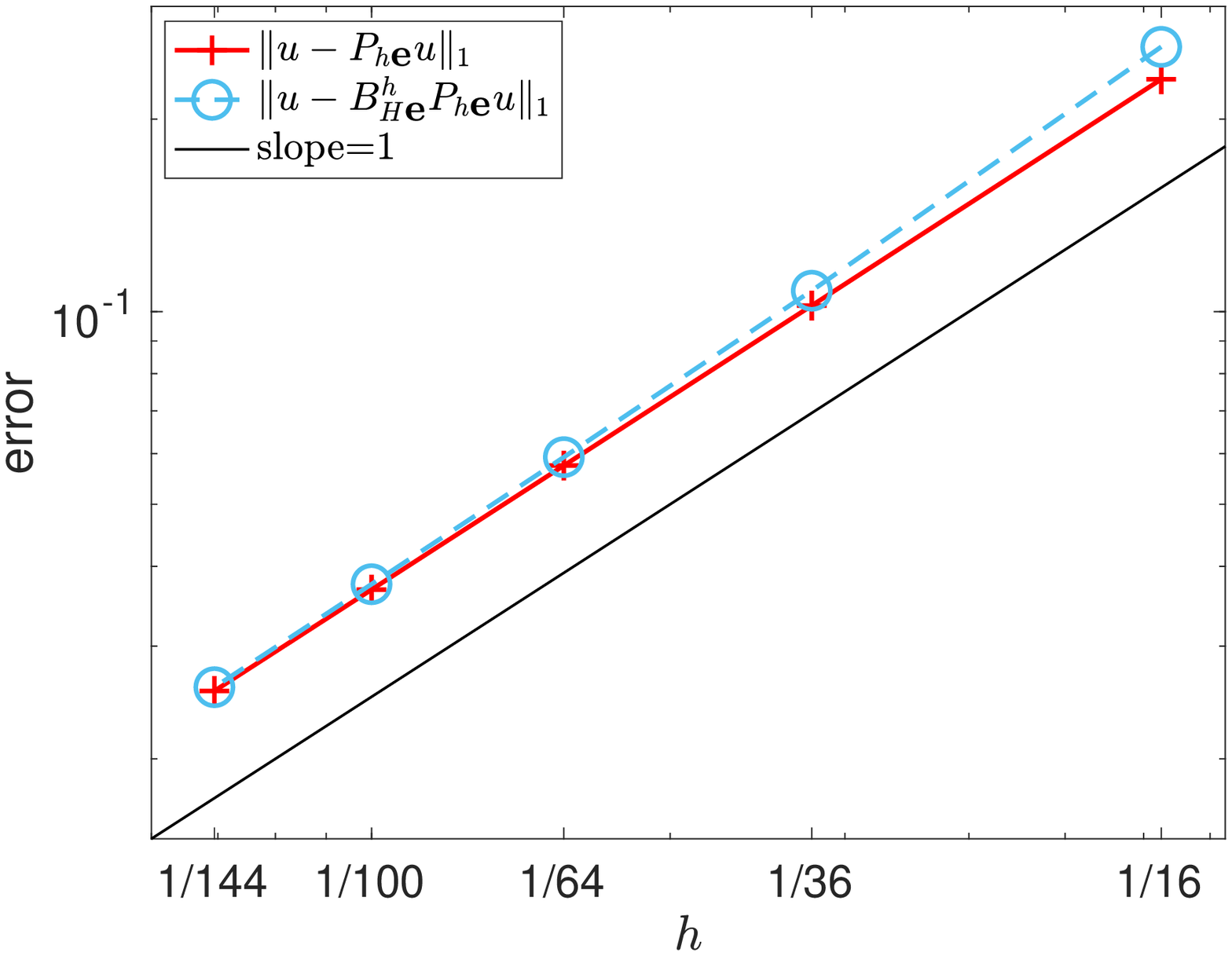}
\end{minipage}
\hfill
\begin{minipage}[t]{0.5\linewidth}
\centering
\includegraphics[width=8cm]{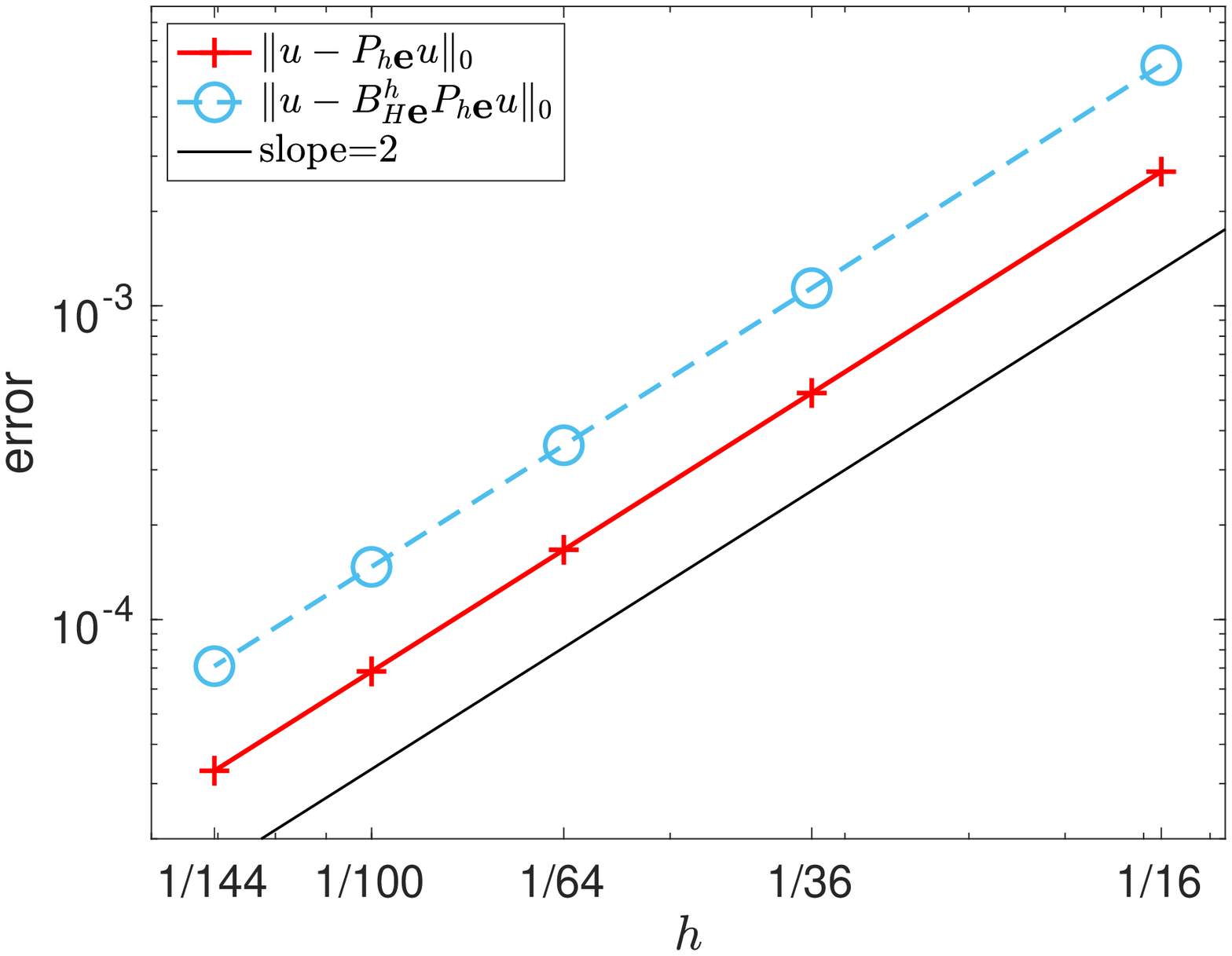}
\end{minipage}
\caption{\footnotesize(Example 2) The convergence curves of solutions in $H^1$ norm and $L^2$ norm by the standard FEM and Algorithm {\ref{alg:bvp_ts}}.}\label{fig:ex3-1}
\end{figure}

\begin{figure}[htbp]
\centering
\includegraphics[width=8cm]{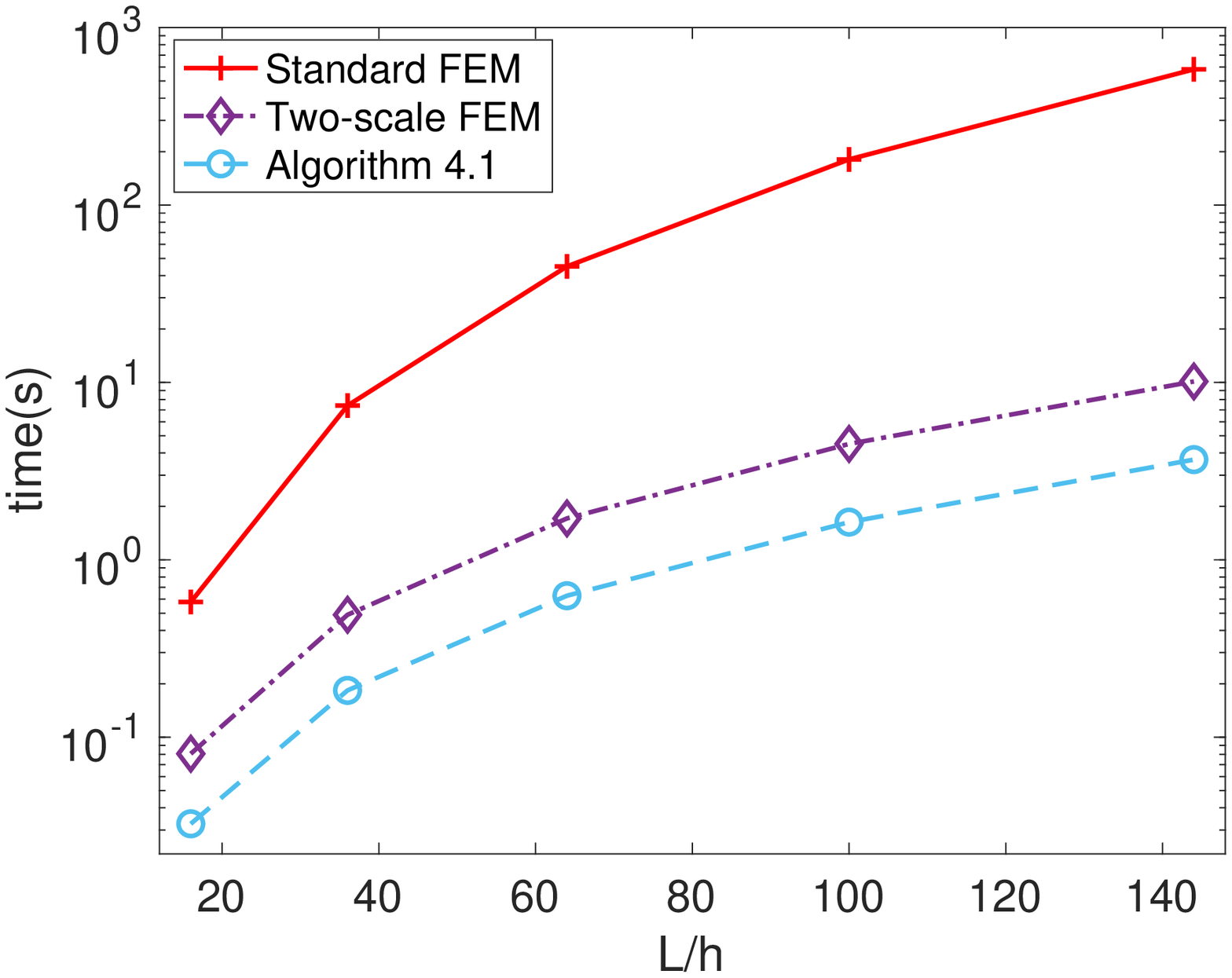}
\caption{\footnotesize(Example 2) The time consumptions of the three methods.}\label{fig:ex3-7}
\end{figure}

\begin{table}[htbp]
\caption{\footnotesize Example 2: The running time (s) by the two-scale FEM.}\label{ex3-3}
  \centering
    \begin{tabular}{|c|c|c|c|c|c|c|}
        \hline
         $\frac{L}{h}\times \frac{L}{H}\times \frac{L}{H}$ & $P_{H\textbf{e}}u$ & $P_{h\textbf{e}_1+H\hat{\textbf{e}}_1}u$ & $P_{h\textbf{e}_2+H\hat{\textbf{e}}_2}u$ & $P_{h\textbf{e}_3+H\hat{\textbf{e}}_3}u$ & \begin{minipage}{1.6cm}\vspace{0mm} \center $B_{H\textbf{e}}^hP_{h\textbf{e}}u$ \vspace{-2mm} \end{minipage} & Total \\
        \hline
        $16\times4\times4$    & 0.0062 & 0.0257 & 0.0243 & 0.0236 & 0.0008 & 0.0806 \\
        \hline
        $36\times6\times6$   & 0.0292 & 0.1510 & 0.1539 & 0.1506 &  0.0052 & 0.4899 \\
        \hline
        $64\times8\times8$   & 0.0677 & 0.5428 & 0.5385 & 0.5349 &  0.0275 & 1.7114 \\
        \hline
        $100\times10\times10$ & 0.1424 & 1.4185 & 1.4174 & 1.4132 &  0.1095 & 4.5010 \\
        \hline
        $144\times12\times12$ & 0.2577 & 3.1907 & 3.1658 & 3.1772 &  0.3328 & 10.1242 \\
        \hline
    \end{tabular}
\end{table}

\begin{table}[htbp]
\caption{\footnotesize Example 2: The running time (s) by Algorithm \ref{alg:bvp_ts}.}\label{ex3-4}
  \centering
    \begin{tabular}{|c|c|c|c|c|c|}
        \hline
         $\frac{L}{h}\times \frac{L}{H}\times \frac{L}{H}$ & $P_{H\textbf{e}}u$ & $P_{h\textbf{e}_1+H\hat{\textbf{e}}_1}u$ & $P_{h\textbf{e}_2+H\hat{\textbf{e}}_2}u,P_{h\textbf{e}_3+H\hat{\textbf{e}}_3}u$ & \begin{minipage}{1.6cm}\vspace{0mm} \center $B_{H\textbf{e}}^hP_{h\textbf{e}}u$ \vspace{-2mm} \end{minipage} & Total \\
        \hline
        $16\times4\times4$    & 0.0062 & 0.0257 & 0.0001 & 0.0005 & 0.0325 \\
        \hline
        $36\times6\times6$   & 0.0292 & 0.1510 & 0.0006 &  0.0028 & 0.1836 \\
        \hline
        $64\times8\times8$   & 0.0677 & 0.5428 & 0.0031 &  0.0137 & 0.6273 \\
        \hline
        $100\times10\times10$ & 0.1424 & 1.4185 & 0.0133 &  0.0570 & 1.6312 \\
        \hline
        $144\times12\times12$ & 0.2577 & 3.1907 & 0.0478 &  0.1748 & 3.6710 \\
        \hline
    \end{tabular}
\end{table}

\textbf{Example 3} Consider a model for the quantum harmonic oscillator:
 \begin{equation*}
 (-\frac{1}{2}\Delta +\frac{1}{2}r^2)u = \lambda u\ \ \ \ in \ \mathbb{R}^3.
 \end{equation*}
 The minimum eigenvalue $\lambda=1.5$ and the associated eigenfunction $u=\pi^{-\frac{3}{4}}\text{e}^{-r^2/2}$. Since the eigenfunctions decay exponentially, in our calculation, the following eigenvalue problem is considered.

\begin{eqnarray*}
    \left\{
       \begin{aligned}
        (-\frac{1}{2}\Delta +\frac{1}{2}r^2)u &= \lambda u\ \ \ \ in \ \Omega, \\
        u &= 0\ \ \ \ \ \ on\  \partial \Omega,
       \end{aligned}
    \right.
\end{eqnarray*}
where $\Omega=[-5.0,5.0]^3$ and $r=(x^2+y^2+z^2)^{\frac 1 2}$. Let $L=10.0$ be the side length of $\Omega$. The numerical results of the minimum eigenvalue and associated eigenfunction are presented in Figures \ref{ex5-1}-\ref{ex5-6} and Tables \ref{ex5-3}-\ref{ex5-4}. Figure \ref{ex5-1} supports the convergence results of Theorem \ref{thm:eig_ts}. For this linear eigenvalue problem with symmetric solution, our results illustrate the efficiency of the symmetrized two-scale finite element method (Algorithm {\ref{alg:eig_ts}}).

\begin{figure}[H]
\vspace{-0.2cm}
\begin{minipage}[t]{0.5\linewidth}
\centering
\includegraphics[width=8cm]{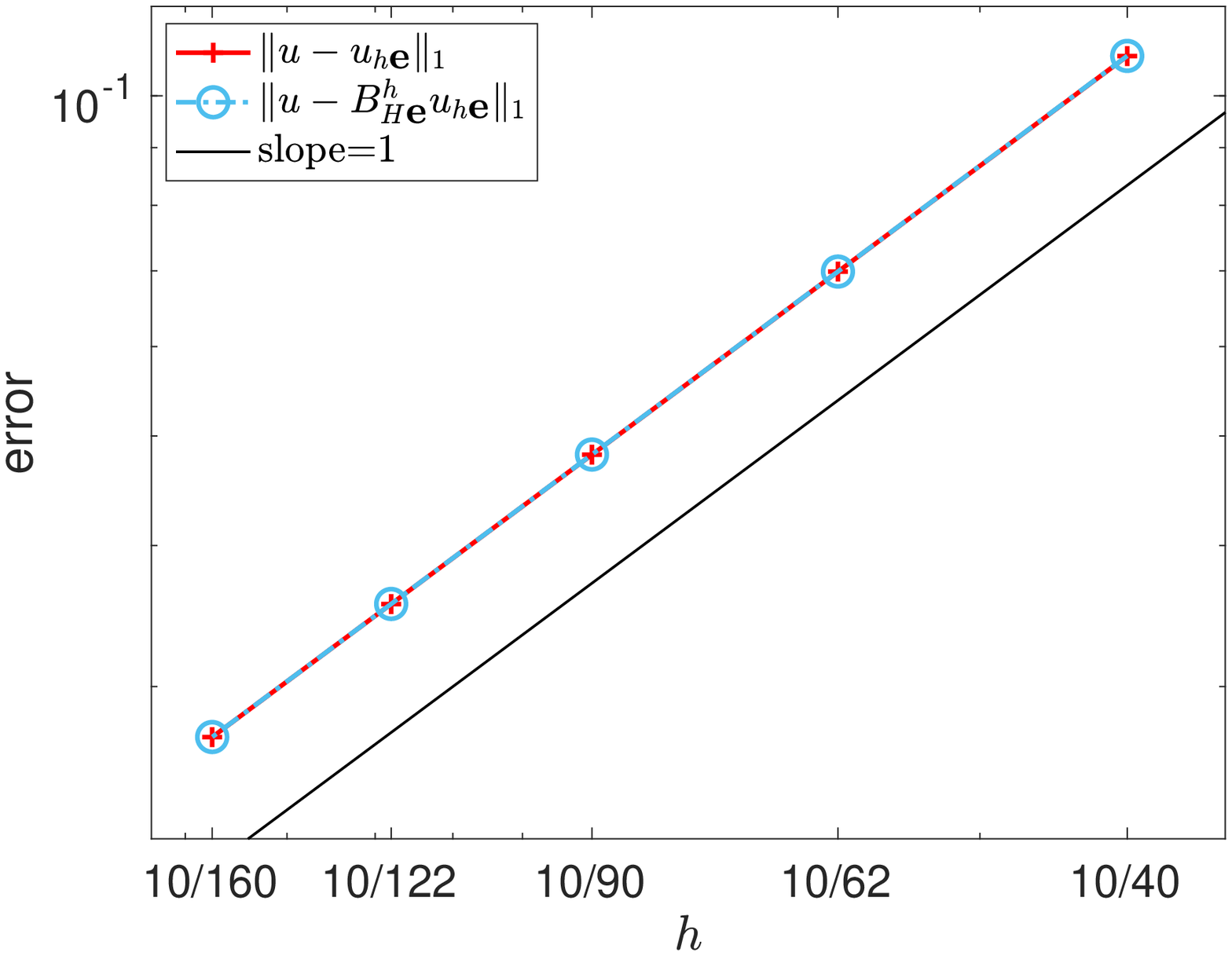}
\end{minipage}
\hfill
\begin{minipage}[t]{0.5\linewidth}
\centering
\includegraphics[width=8cm]{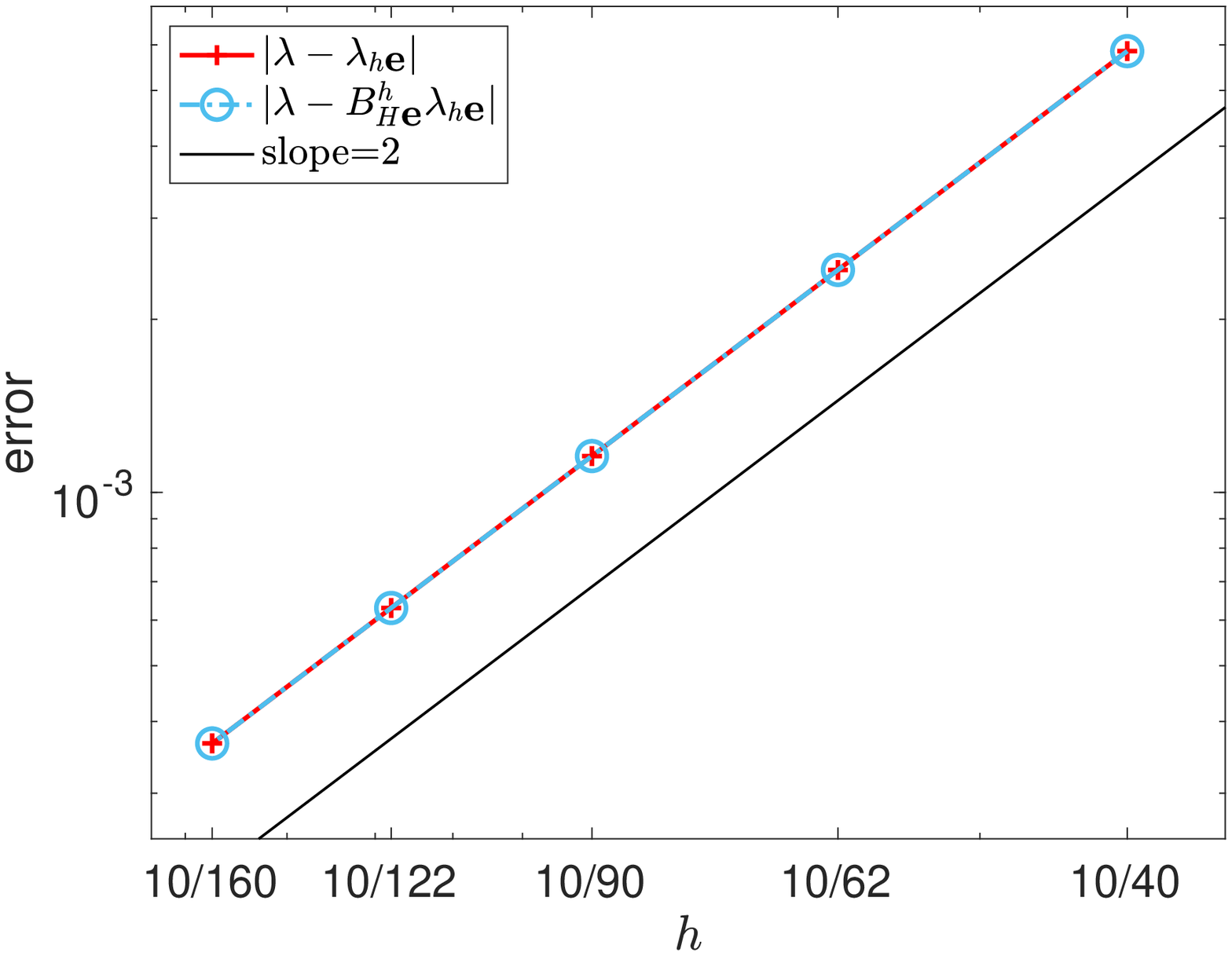}
\end{minipage}
\caption{\footnotesize(Example 3) The convergence curves of eigenvalue and eigenfunction by the standard FEM and Algorithm {\ref{alg:eig_ts}}.}\label{ex5-1}
\end{figure}

\begin{figure}[H]
\centering
\includegraphics[width=8cm]{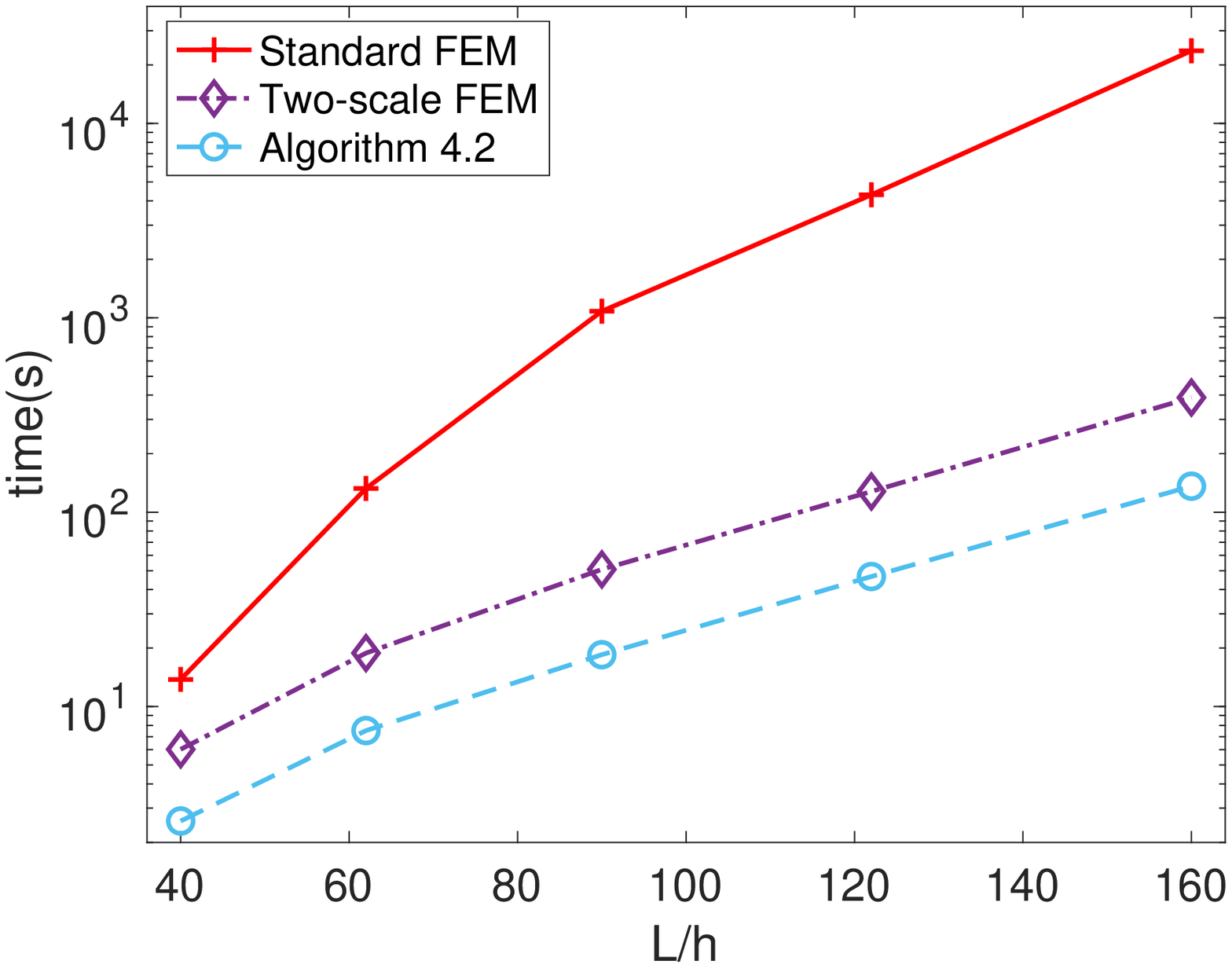}
\caption{\footnotesize(Example 3) The time consumptions of the three methods.}\label{ex5-6}
\end{figure}

\begin{table}[H]
\caption{\footnotesize Example 3: The running time (s) by the two-scale FEM.}\label{ex5-3}
  \centering
    \begin{tabular}{|c|c|c|c|c|c|c|c|}
        \hline
         $\frac{L}{h}\times \frac{L}{H}\times \frac{L}{H}$ & $u_{H\textbf{e}}$ & $u_{h\textbf{e}_1+H\hat{\textbf{e}}_1}$ & $u_{h\textbf{e}_2+H\hat{\textbf{e}}_2}$ & $u_{h\textbf{e}_3+H\hat{\textbf{e}}_3}$ & \begin{minipage}{1.5cm}\vspace{1mm} \center $B_{H\textbf{e}}^hu_{h\textbf{e}}$ \vspace{1mm} \end{minipage} & $B_{H\textbf{e}}^h\lambda_{h\textbf{e}}$ & Total \\
        \hline
        $40\times20\times20$ & 0.7060 & 1.7362  & 1.6241  & 1.8173  & 0.1405 & 0.0001 & 6.0242 \\
        \hline
        $62\times25\times25$ & 1.6529 & 5.3517  & 5.1389  & 6.1617  & 0.5373 & 0.0001 & 18.8426 \\
        \hline
        $90\times30\times30$ & 3.5823 & 13.2750  & 13.4933  & 18.8052  & 1.6285 & 0.0001 & 50.7844 \\
        \hline
        $122\times35\times35$ & 7.4273 & 35.3173 & 41.0032  & 40.0038 & 4.0751 & 0.0001 & 127.8268 \\
        \hline
        $160\times40\times40$ & 13.1982 & 113.4475 & 105.7242 & 146.7036 & 9.8352 & 0.0016 & 388.9103 \\
        \hline
    \end{tabular}
\end{table}

\begin{table}[H]
\caption{\footnotesize Example 3: The running time (s) by Algorithm \ref{alg:eig_ts}.}\label{ex5-4}
  \centering
    \begin{tabular}{|c|c|c|c|c|c|c|}
        \hline
        $\frac{L}{h}\times \frac{L}{H}\times \frac{L}{H}$ & $u_{H\textbf{e}}$ & $u_{h\textbf{e}_1+H\hat{\textbf{e}}_1}$ & $u_{h\textbf{e}_2+H\hat{\textbf{e}}_2},u_{h\textbf{e}_3+H\hat{\textbf{e}}_3}$ & \begin{minipage}{1.5cm}\vspace{1mm} \center $B_{H\textbf{e}}^hu_{h\textbf{e}}$ \vspace{1mm} \end{minipage} & $B_{H\textbf{e}}^h\lambda_{h\textbf{e}}$ & Total \\
        \hline
        $40\times20\times20$ & 0.7060 & 1.7362  & 0.0009  & 0.1338 & 0.0001 & 2.5770 \\
        \hline
        $62\times25\times25$ & 1.6529 & 5.3517  & 0.0029  & 0.5145 & 0.0001 & 7.5221 \\
        \hline
        $90\times30\times30$ & 3.5823 & 13.2750  & 0.0118 & 1.6162 & 0.0001 & 18.4854 \\
        \hline
        $122\times35\times35$ & 7.4273 & 35.3173 & 0.0341 & 3.8900 & 0.0001 & 46.6688 \\
        \hline
        $160\times40\times40$ & 13.1982 & 113.4475 & 0.0953 & 9.3284 & 0.0003 & 136.0697 \\
        \hline
    \end{tabular}
\end{table}

\textbf{Example 4} Consider the Kohn-Sham equation for the helium atom:
\begin{equation*}
        \left(-\frac{1}{2} \Delta  - \frac{2}{|r|} + \int \frac{\rho(r')}{|{r}- {r}'|} d {r}' + V_{xc}\right) u = \lambda u\ \ \ in\ \mathbb{R}^3,
\end{equation*}
with $\|u\|_{0,\mathbb{R}^3}=1$, where $|r|=(x^2+y^2+z^2)^{\frac 1 2}$ and $\rho=2|u|^2$. In our computation, we solve the following nonlinear eigenvalue problem \cite{Dai2008}: find $(\lambda, u)\in \mathbb{R}\times H^1_0(\Omega)$ such that $\|u\|_{0,\Omega}=1$ and
\begin{eqnarray}\label{eq:helium}
    \left\{
    \begin{aligned}
         \left(-\frac{1}{2} \Delta  -\frac{2}{|r|} + \int_\Omega \frac{\rho(r')}{|{r}- {r}'|} d {r}' + V_{xc}\right) u &= \lambda u\ \ \ in\ \Omega,\\
        u &= 0\ \ \ \ \ on \ \partial \Omega,
    \end{aligned}
    \right.
\end{eqnarray}
where $\Omega = [-5.0,5.0]^3$. We choose $V_{xc}(\rho)=-\frac 3 2 \alpha (\frac 3 \pi \rho)^{1/3}$ with $\alpha = 0.77298$. Let $L=10.0$ be the side length of $\Omega$.

For this nonlinear eigenvalue problem, it has been shown in \cite{hou2021} that the two-scale finite element method is more economic than the standard finite element method. Because the first eigenfunction of \eqref{eq:helium} is symmetric, the two-scale finite element discretization proposed in \cite{hou2021} can be also combined with Algorithm \ref{alg:symmetric-dd}. To illustrate this observation, numerical results are presented in Figure \ref{fig:ex7-7} and Table \ref{ex7-6}. One can see that by combining the idea of Algorithm \ref{alg:symmetric-dd} with the two-scale finite element method, the time consumption is reduced further.

\begin{figure}[H]
\vspace{-0.2cm}
\centering
\includegraphics[width=8cm]{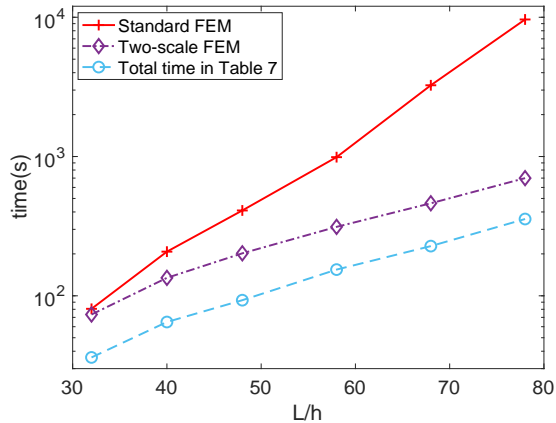}
\caption{\footnotesize (Example 4) The time consumptions of the three methods.}\label{fig:ex7-7}
\end{figure}
%

\begin{table}[H]
\caption{\footnotesize Example 4: The running time (s) by combining the two-scale FEM \cite{hou2021} with Algorithm \ref{alg:symmetric-dd}.}\label{ex7-6}
  \centering
    \begin{tabular}{|c|c|c|c|c|c|c|c|}
        \hline
         $\frac{L}{h}\times \frac{L}{H}\times \frac{L}{H}$ & $u_{H\textbf{e}} $ & $u_{h\textbf{e}_1+H\hat{\textbf{e}}_1}$ & $u_{h\textbf{e}_2+H\hat{\textbf{e}}_2},u_{h\textbf{e}_3+H\hat{\textbf{e}}_3}$ & \begin{minipage}{1.5cm}\vspace{1mm} \centering $u_{H\textbf{e}}^h$ \vspace{1mm} \end{minipage} & $\lambda^h_{H\textbf{e}}$ &  Total \\
        \hline
        $32\times18\times18$ & 9.62  & 18.71  & 0.0006  & 0.16 & 7.55 & 36.10 \\
        \hline
        $40\times20\times20$ & 15.11  & 34.77  & 0.0012 & 0.29 & 14.55 & 64.84 \\
        \hline
        $48\times22\times22$ & 19.87  & 46.39  & 0.0015 & 0.54 & 26.05 & 93.00 \\
        \hline
        $58\times24\times24$ & 27.12  & 79.33  & 0.0028 & 0.91 & 46.61 & 154.25 \\
        \hline
        $68\times26\times26$ & 36.50  & 112.12 & 0.0042 & 1.47 & 76.90 & 227.41 \\
        \hline
        $78\times28\times28$ & 49.10  & 187.78 & 0.0062 & 2.22 & 116.20 & 355.92 \\
        \hline
    \end{tabular}
\end{table}

\section{Conclusions}
In this paper, a symmetrization algorithm is developed to perform a vector transformation associated with symmetric functions. By combining the symmetrization algorithm with the two-scale finite element methods, some symmetrized two-scale finite element methods are proposed for the elliptic source and eigenvalue problems with symmetric solutions. Numerical analyses and numerical results show that the symmetrized two-scale finite element methods save computational cost significantly compared with the corresponding two-scale finite element methods. The symmetrization algorithm can also be combined with the postprocessed two-scale finite element method to obtain efficient numerical methods \cite{LSZ11,LZ16}. Besides, some two-scale finite element discretizations have been developed for a class of integral equation and integrodifferential equation \cite{chen2011,Liu2007,Xu2004}, which can be also combined with the symmetrization algorithm to get more efficient algorithms. Moreover, we believe that our symmetrization algorithm is significant for solving higher dimensional problems on tensor product domains.

\section*{Acknowledgments}
The authors thank Professor Huajie Chen for enlightening discussions. P. Hou was partially supported by the National Natural Science Foundation of China (grant 11971066). F. Liu was partially supported by the National Natural Science Foundation of China (grant 11771467) and the disciplinary funding of Central University of Finance and Economics. A. Zhou was partially supported by the National Key R \& D Program of China under grants 2019YFA0709600 and 2019YFA0709601.

\bibliographystyle{plain}
\bibliography{references}
\end{document}